\documentclass[a4paper]{amsart}

\usepackage[english]{babel}
\usepackage{amssymb}
\usepackage{hyperref}
\usepackage{tikz}
\usepackage[margin=10pt,font=small,labelfont=bf,labelsep=endash]{caption}
\usetikzlibrary{matrix}
\numberwithin{equation}{section}
\theoremstyle{plain}
\newtheorem{lemma}{Lemma}[section]
\newtheorem{theorem}[lemma]{Theorem}
\newtheorem*{theorem*}{Theorem}
\newtheorem{proposition}[lemma]{Proposition}
\newtheorem{corollary}[lemma]{Corollary}

\newtheorem{theorem-var}{Theorem}[]
\newtheorem{corollary-var}{Corollary}[]

\theoremstyle{definition}
\newtheorem{definition}[lemma]{Definition}
\newtheorem*{definition*}{Definition}

\theoremstyle{remark}
\newtheorem{remark}[lemma]{Remark}

\DeclareMathOperator{\Cent}{Cent}
\DeclareMathOperator{\Span}{span}

\DeclareMathOperator{\height}{ht}
\DeclareMathOperator{\supp}{supp}
\DeclareMathOperator{\Ort}{Ort}

\newcommand{\calA}{\mathcal A} 
\newcommand{\calB}{{\mathcal B}} 
\newcommand{\calC}{{\mathcal C}}

\newcommand{\calS}{\mathcal S} \newcommand{\calR}{\mathcal R}

\newcommand\half{\frac{1}{2}}

 \newcommand{\mN}{\mathbb N}
 
\newcommand{\mZ}{\mathbb Z}

\newcommand\bb{\mathfrak b}
\newcommand\D{\Phi}
\renewcommand\l{\lambda}
\newcommand\Dp{\Phi^+}

\newcommand\Da{\widehat\Phi}
\newcommand\Pia{{\widehat\Pi}}
\newcommand\Dap{\widehat\Phi^+}
\newcommand\Wa{\widehat{W}}
\renewcommand\d{\delta}

\renewcommand\a{\alpha}
\renewcommand\aa{\mathfrak a}
\renewcommand\b{\bb}
\newcommand\ganz{\mathbb Z}

\newcommand\g{\mathfrak g}
\newcommand{\gog}{\mathfrak g}
\newcommand{\gol}{\mathfrak l}

\newcommand{\gou}{\mathfrak u}

\newcommand{\gob}{\mathfrak b}
\newcommand{\gop}{\mathfrak p}

\newcommand{\h}{\mathfrak t}
\newcommand\nat{\mathbb N}
\newcommand\s{\sigma}
\newcommand\si{\sigma}
\renewcommand\h{\mathfrak h}
\newcommand\ov{\overline}
\newcommand{\Wab}{\mathcal W_\s^{ab}}
\renewcommand{\S}{\Sigma}
\newcommand\C{\mathbb C}

\renewcommand\k{\mathfrak k}
\newcommand\be{\beta}

\newcommand\ha{\widehat{\mathfrak h}}

\newcommand{\gra}{\alpha} \newcommand{\grb}{\beta}    \newcommand{\grg}{\gamma}
\newcommand{\grd}{\delta} \newcommand{\grl}{\lambda}  \newcommand{\grs}{\sigma}
\newcommand{\gre}{\varepsilon} 

\newcommand{\grG}{\Gamma}

 \newcommand{\grS}{\Sigma}

\newcommand{\lra}        {\longrightarrow}

\newcommand{\vuoto}      {\varnothing}

\renewcommand{\ge}      {\geqslant}
\renewcommand{\geq}      {\geqslant}
\renewcommand{\le}      {\leqslant}
\renewcommand{\leq}      {\leqslant}

\newcommand{\ol}         {\overline}

\renewcommand{\setminus}      {\smallsetminus}

\newcommand{\st}       {\, | \,}

\newcommand{\mru}       {\mathrm u}

\DeclareMathOperator{\rk}{rk}

\begin{document}

\title[Spherical nilpotent orbits and abelian subalgebras]{Spherical nilpotent orbits and abelian subalgebras in isotropy representations}
\author[J.~Gandini]{Jacopo Gandini}
\author[P.~M\"oseneder Frajria]{Pierluigi M\"oseneder Frajria}
\author[P.~Papi]{Paolo Papi}
%\date{\today}

\begin{abstract} 

 Let $G$ be a simply connected semisimple algebraic group with Lie algebra $\g$, let $G_0 \subset G$ be the symmetric subgroup defined by an algebraic involution $\s$ and let $\g_1 \subset \g$ be the isotropy representation of $G_0$. Given an abelian subalgebra $\aa$ of $\g$ contained in $\g_1$ and stable under the action of some Borel subgroup $B_0 \subset G_0$, we classify the $B_0$-orbits in $\aa$ and we characterize the sphericity of $G_0 \aa$. Our main tool is the combinatorics of $\s$-minuscule elements in the affine Weyl group of $\g$ and that of strongly orthogonal roots in Hermitian symmetric spaces.
\end{abstract}
\keywords{Spherical variety, abelian ideal , isotropy representations, nilpotent orbit}
\subjclass[2010]{Primary    17B20; Secondary 14M27, 17B08}
\maketitle

\section{Introduction}
%%%%%%%%%%%%%%%%%%%%%%%%%%%%%%%%%%%%%%%%%%%%%%%%%%%%%%%
%%%%%%%%%%%%%%%%%%%%%%%%%%%%%%%%%%%%%%%%%%%%%%%%%%%%%%%
Let $G$ be a connected simply connected semisimple complex algebraic group with Lie algebra $\g$. Let $B$ be a Borel subgroup, and set 
$\gob= {\rm Lie} B$. Recall that a $G$-variety $X$ is called $G$-\textit{spherical} if it possesses an open $B$-orbit. The relationships between spherical nilpotent orbits and abelian ideals of $\gob$ have been first investigated in \cite{PR}.
There it is shown that if $\aa$ is an abelian ideal of $\gob$, then any nilpotent orbit meeting $\aa$ is a $G$-spherical variety
and $G\aa$ is the closure a spherical nilpotent orbit. In particular, $B$ acts on $\aa$ with finitely many orbits.\par
Subsequently, Panyushev \cite{Pa4} dealt with similar questions in the $\mathbb Z_2$-graded case. Let $\sigma$ be an involution of $G$ and 
$\gog = \gog_0 \oplus \gog_1$ be the corresponding eigenspace decomposition at the Lie algebra level.
Let $G_0$ be the connected subgroup of $G$ corresponding to $\gog_0$ and $B_0\subset G_0$ a Borel subgroup of $G_0$.
The ``graded'' analog of the set of abelian ideals of $\gob$ is the set $\mathcal I_{ab}^\s$ of (abelian) $B_0$-stable subalgebras of $\gog_1$.
\begin{definition} We say that $\aa\in\mathcal I_{ab}^\s$ is {\it $G$-spherical} (resp. {\it $G_0$-spherical}) if all orbits $Gx,\,x\in\aa$ are $G$-spherical
(resp.  if all orbits $G_0 x,\,x\in\aa$ are $G_0$-spherical).
\end{definition}
\par
Panyushev \cite{Pa2} started  the classification of the spherical nilpotent $G_0$-orbits in $\g_1$. The classification of the spherical nilpotent $G_0$-orbits in $\gog_1$ was then completed by King \cite{king} (see also \cite{BCG}, where the classification is reviewed and a missing case is pointed out). Shortly afterwards, Panyushev \cite{Pa4} noticed the emergence of non-spherical subalgebras $\aa\in\mathcal I_{ab}^\s$, and classified the involutions $\s$ for which these subalgebras exist. After explicit verifications, he also noticed that an element $\aa\in\mathcal I_{ab}^\s$ is $G$-spherical if and only if it is $G_0$-spherical, but no verification was given, and no conceptual proof was known.

%Panyushev noticed the emergence of non-spherical subalgebras and classified the involutions for which these
%subalgebras exist. He also noticed that $\aa$ is $G$-spherical if and only  if it is $G_0$-spherical, but no conceptual proof was given for the non-trivial implication in this equivalence.\par

The purpose of the present paper is to deepen and expand the results quoted above in the following  directions. Let $\aa\in\mathcal I_{ab}^\s$.
\begin{enumerate}
\item[i)]    
%We show that if $\aa \in \mathcal I_{ab}^\s$ and $x \in \aa$, then $\height_0(x) \leq 3$ and $\height_1(x) \leq 4$ (Corollary \ref{73}). These properties completely characterize the elements of abelian subalgebras of $\g_1$ which are stable under some Borel subgroup of $G_0$ (see Section \ref{51}). As a corollary, using Panyushev criterions for the $G_0$-sphericity of a nilpotent element in $\g_1$, it follows that $\aa \in \mathcal I_{ab}^\s$ is $G_0$-spherical if and only if it is $G$-spherical.

We clarify the connections  between $G_0$-orbits of nilpotent elements in $\gog_1$, spherical  $G$-orbits of nilpotent elements in $\g_1$ and $G_0$-orbits of abelian subalgebras in $\gog_1$ which are stable under some Borel subalgebra of $\gog_0$.
\item[ii)] We prove that $B_0$ acts on $\aa$ with finitely many orbits, independently of its sphericity. Moreover,  we parametrize
orbits via orthogonal set of weights of $\aa$.
\item[iii)]  Assume that there exist non-spherical subalgebras. We give a construction of a canonical non-spherical subalgebra $\aa_p$.
\item[iv)]  We give a simple criterion to decide whether $\aa$ is spherical or not: in  Theorem \ref{MT}
 we show that there exists $\ov \aa\in \mathcal I_{ab}^\s$ such that 
$\aa$ is non-spherical if and only if $\aa\supset \ov \aa$.
 \end{enumerate}
%A more detailed account of  results on sphericity (as Theorem \ref{13}) is given  in Section \ref{s2}: there we also explain how to obtain  an almost 
%inspection-free proof of Panyushev's classification theorem (Theorem \ref{Panyushev}) and a conceptual explanation of the coincidence between $G$-sphericity and 
%$G_0$-sphericity. \par
One  important feature of our approach lies in the methods used. The theory of abelian ideals and its graded version rely on a strict relationship with the geometry of alcoves 
of the affine Weyl group $\Wa$ of $\g$ \cite{K}, \cite{CP}, and, for the graded case, with Kac's classification of finite order automorphisms of semisimple Lie algebras \cite{Kac}, \cite{CMP}, \cite{CMPP}.

The main link is that a $B_0$-stable subalgebra $\aa$ can be encoded by an element   $w_\aa\in\Wa$ defined through its  set  of inversions $N(w_\aa)$ (cf. (\ref{nw})). 
The elements so obtained, called 
$\s$-minuscule  (Definition \ref{sm}), pave a convex polytope in the dual space of a Cartan subalgebra of $\g$ and have remarkable properties: see Section \ref{bs} for a recollection of these facts.
It has been explicitly asked (e.g., in \cite{Pa4}) to use the above connections as a tool for dealing with problems about sphericity. This is what we do here.
%\end{document}
%, by clarifying the role of the affine root system of $\gog$ in the sphericity of $\mathcal I^\s_{ab}$.

We start discussing items i)-iv) by making  the content of i) more precise. Define the \textit{height} of a nilpotent element $x \in \g$ as
$$
	\height(x) = \max\{n \in \mN \st \mathrm{ad}(x)^n \neq 0\}.
$$
In the adjoint case, Panyushev \cite{Pa2} completely characterized the spherical nilpotent $G$-orbits in $\g$ by showing that, for $x \in \g$, the orbit $Gx$ is spherical if and only if $\height(x) \leq 3$. Subsequently, Panyushev and R\"ohrle \cite{PR} proved that, if $\aa \subset \b$ is an abelian ideal, then the saturation $G\aa$ is spherical. On the other hand, if $Gx$ is spherical, by chosing $\b$ properly it is always possible to construct an abelian ideal $\aa \subset \b$ such that $G\aa = \ol{Gx}$. Therefore we may regard both these properties as consequences of the small height of the nilpotent element $x$.

For $i=0,1$ define the $i$-height of a nilpotent element $x \in \g_1$ as
$$
	\height_i(x) = \max\{n \in \mN \st \mathrm{ad}(x)^n_{|\g_i} \neq 0\}.
$$
In \cite{Pa2} Panyushev showed  that, for $x \in \g_1$, the following implications hold
$$
	\height(x) \leq 3 \Longrightarrow G_0 x \mbox{ spherical} \Longrightarrow \height_0(x) \leq 4,  \height_1(x) \leq 3.
$$
%The classification of the spherical nilpotent $G_0$-orbits in $\gog_1$ was then completed by King \cite{king} (see also \cite{BCG}, where the classification is reviewed and some missing cases are pointed out). Shortly afterwards, Panyushev \cite{Pa4} noticed the emergence of non-spherical subalgebras $\aa\in\mathcal I_{ab}^\s$, and classified the involutions $\s$ for which these subalgebras exist. After explicit verifications, he also noticed that an element $\aa\in\mathcal I_{ab}^\s$ is $G$-spherical if and only if it is $G_0$-spherical, but no verification was given, and no conceptual proof was known.
In Corollary \ref{73}, we show the following result.
\begin{theorem*}
 If $\aa \in \mathcal I_{ab}^\s$ and $x \in \aa$, then $\height_0(x) \leq 3$ and $\height_1(x) \leq 4$.
 \end{theorem*}
  These properties completely characterize the elements of abelian subalgebras of $\g_1$ which are stable under some Borel subgroup of $G_0$ (see Section \ref{51}). As a corollary, using Panyushev's criterion for the $G_0$-sphericity of a nilpotent element in $\g_1$, it follows that $\aa \in \mathcal I_{ab}^\s$ is $G_0$-spherical if and only if it is $G$-spherical.
\par
Regarding ii), a well known result independently due to Brion \cite{brion} and Vinberg \cite{Vi1}, states that every spherical $G$-variety
%, which can be defined in the same way for any reductive group, 
contains finitely many $B$-orbits. 
In particular, every abelian ideal $\aa$ of $\b$ contains finitely many $B$-orbits. In  \cite{Pa5}, the same result has been proved avoiding the use of the sphericity of $G\aa$. 
In Section \ref{s4},  we prove, along the same lines, the finiteness theorem quoted in ii), in the more general context of finite order automorphisms of $G$ (see Theorem \ref{teo:B0-orbite}).\par

%In  \cite{Pa4} the involution $\s$ is said to be of \textit{type I} if every $\aa \in \mathcal I_{ab}^\s$ is $G_0$-spherical, and of \textit{type II} otherwise.
To streamline our approach to  iii),  recall that involutions of $\g$ are encoded by the datum of one or two simple roots (with suitable features) in the extended Dynkin diagram $\Pia$ of $\g$. The main result of  \cite{Pa4} has been  rephrased by Panyushev  as follows:  there exists a non spherical $\aa\in\mathcal I^\s_{ab}$ if and only if if $\s$ is defined by a single simple root $\a_p$, which is long and non-complex (see Definition \ref{nc}). (As usual, if $\widehat \Pi$ is simply laced, every root is regarded as long).
%\footnote{Mi pare che nella formulazione di Panyushev le radici medie anche danno luogo a involuzioni di tipo II, \`e in accordo con la nostra formulazione? Forse sarebbe meglio dare la formulazione di Panyushev nell'introduzione che suona pi\`u immediata.}. 
However this claim was obtained  as a by-product of direct considerations on various classes of involutions and by constructing case-by-case a non-spherical element $\aa \in \mathcal I_{ab}^\s$ for all involution satisfying the previous  condition.

In this paper we observe that, precisely  when $\a_p$ is long and non-complex,  there exists a {\sl special} Êelement $\aa_{p}\in \mathcal I_{ab}^\s$, which plays a role in the classification of maximal elements in $\mathcal I_{ab}^\s$ performed in \cite{CMPP}.
In Section \ref{5} we study the properties of $\aa_{p}$. In particular, using the combinatorics of $N(w_{\aa_{p}})$, we prove that $\aa_p$  is 
 not $G_0$-spherical. The method is combinatorial:  we associate to any orthogonal set of maximal cardinality in $N(w_{\aa_{p}})\setminus\{\a_p\}$ a generalized Cartan matrix of affine type, whose type is severely restricted (see Proposition \ref{63}). The information we obtain from this Cartan matrix allows us to build up a generic element $x \in \aa_{p}$ with $\height_1(x) = 4$, proving  that $\aa_{p}$ is not $G_0$-spherical.
 
The same strategy is applied in a wider context in Section \ref{cla}, and it enables us to classify the spherical elements of $\mathcal I_{ab}^\s$, as outlined in iv). 
%Recall that the complexity of an irreducible $G_0$-variety is defined as
%$$
%c_{G_0}(X) = \min_{x\in X}\text{codim} B_0 x.
%$$
%In particular, the spherical $G_0$-varieties coincide with the $G_0$ varieties of complexity zero. Using a formula for the complexity given in \cite{Pa2}, in Corollary \ref{complexity} we show  that $c_{G_0}(G_0\aa_p)=1$.
The construction of the minimal non-spherical subalgebra $\ov\aa$  is  based on the combinatorics of strongly orthogonal roots in Hermitian symmetric spaces. Many related  technical results might be of independent interest, and they are displayed in Section \ref{4}. To construct  $\ov\aa$,
decompose $\Pia\setminus\{\a_p\}$ into a disjoint union of connected components $\S$. Then in each $\S$ there exists a unique simple root $\a_\S$ non-orthogonal to $\a_p$,
and it turns out that $\a_\S$ determines an Hermitian involution of tube type of the Lie algebra $\g_\S$ having $\S$ as set of simple roots (see Subsection \ref{tt} and Proposition \ref{tubet}). If $\Phi(\S)^+_1$ denotes the set of positive roots of $\g_\S$ having $\a_\S$ in their support, in Lemma \ref{lemmahermitiano} we prove that there exists a unique subset  $\mathcal A_\S$ which is an antichain in $\Phi(\S)^+_1$ w.r.t.  the dominance order $\leq_\grS$ defined by $\grS$ and which is a maximal orthogonal subset of $\Phi(\S)^+_1$.
Next, we prove that $\bigcup_\S \bigcup\limits_{\eta\in\mathcal A_\S}\{\xi+\a_p\mid \xi\leq_\grS \eta\}$  is the set of inversions of a $\s$-minuscule element, hence it determines  an element $\ov\aa\in \mathcal I_{ab}^\s$, which turns out to have the property described in iv).

%In the present paper we clarify the connections, appearing in Corollary \ref{co}, between $G_0$-orbits of nilpotent elements in $\gog_1$, spherical  $G$-orbits of nilpotent elements in $\g_1$ and $G_0$ orbits of abelian subalgebras in $\gog_1$ which are stable under some Borel subalgebra of $\gog_0$. 
%Regarding (iv), it is not hard to see that a spherical orbit $Ge, e\in \g_1$, intersects densely the $G_0$-orbit of an abelian subalgebra stable under a suitable Borel subalgebra $\gob_0$ (see  Corollary \ref{co}). Conversely, we  show that the heights of elements of such subalgebras are severely constrained, even though the corresponding orbits need not to be spherical. In Section \ref{cla} we prove that if 
%$\aa$ is a abelian $B_0$-stable subalgebra in $\gog_1$ and  $e \in \aa$ is  such that $\ol{G_0e} = G_0 \aa$,  then 
%\begin{equation}\label{constraints}
%\height(e) \leq 4\text{ and }\gog_0(4) = \{0\}
%\end{equation}
%(notation as in \eqref{g0i}). In particular, $G_0\aa$ is spherical if and only if $G\aa$ is spherical.
%Constraints  \eqref{constraints} completely characterize generic elements in abelian subalgebras of $\g_1$ stable under some Borel subalgebra: see Corollary \ref{73}. \\

\noindent{\it Acknowledgements.} We thank Dmitri Panyushev for useful discussions.

%%%%%%%%%%%%%%%%%%%%%%%%%%%%%%%%%%%%%%%%%%%%%%%%%%%%%%%
%%%%%%%%%%%%%%%%%%%%%%%%%%%%%%%%%%%%%%%%%%%%%%%%%%%%%%%
\section{Setup}
%%%%%%%%%%%%%%%%%%%%%%%%%%%%%%%%%%%%%%%%%%%%%%%%%%%%%%%
%%%%%%%%%%%%%%%%%%%%%%%%%%%%%%%%%%%%%%%%%%%%%%%%%%%%%%%

Let $G$ be a semisimple, connected and simply connected complex algebraic group with Lie algebra $\g$, and let $B \subset G$ be a Borel subgroup with Lie algebra $\b$. Throughout the  paper,  $\s: G \lra G$ will be an indecomposable automorphism of finite order $m$. Then $\s$ induces an automorphism of $\g$ as well,  still denoted by $\s$. Fix a primitive $m^{\mathrm{th}}$-root of unity $\zeta$ and consider the corresponding $\mZ_m$-grading
$$
	\gog = \bigoplus_{ i \in \mZ_m} \gog_i,
$$
where $\gog_i$ denotes the eigenspace of $\grs$ of weight $\zeta^i$. Then $\gog_0$ is a reductive subalgebra of $\gog$ (see \cite[Lemma 8.1]{Kac}), and the connected reductive subgroup $G_0 \subset G$ defined by $\g_0$ coincides with the set of fixed points of $\s$. Fix a Cartan subalgebra $\h_0 \subset \g_0$, which is abelian since $\g_0$ is reductive. If $\aa\subset \g$ is a $\h_0$-stable subspace, we let $\Psi(\aa)$ denote its set of $\h_0$-weights and, for $\l\in\Psi(\aa)$, we let $\aa^\l$ be the corresponding weight space. \par

Every eigenspace $\gog_i$ is a $G_0$-module under the restriction of the adjoint action. If $i \in \mZ _m$, we denote by $\Phi_i$ the set of the non-zero $\h_0$-weights in $\gog_i$.
%, and if $\gra \in \Phi_i \cup \{0\}$ we denote by $\gog_i^\gra$ the corresponding weight space. 
Denote finally $\Phi=\cup_i\Phi_i$ the set of non-zero weights. We say $\mu,\nu\in\Phi$ are {\it strongly orthogonal} if $(\mu,\nu)=0$ and $\mu\pm\nu\notin \Phi$. (This definition agrees with the usual notion of strongly orthogonal roots in a semisimple Lie algebra). %\footnote{Non basta $\mu\pm\nu\notin \Phi$?} \par

Observe that $\Phi_0$ is the set of $\h_0$-roots for $\g_0$.
As shown in \cite[Chapter 8]{Kac},  $\h_0$ contains a regular element $h_{reg}$ of $\g$. In particular the centralizer $\Cent(\h_0)$ of $\h_0$ in $\g$ is a Cartan subalgebra of $\g$ and $h_{reg}$ defines a set of positive roots in the set of roots of $(\g, \Cent(\h_0))$ and a set $\Dp_0$ of positive roots in $\D_0$. We let $\Pi_0$ be the corresponding set of simple roots, $\b_0$ the corresponding Borel subalgebra, and  $B_0\subset G_0$ the corresponding Borel subgroup.

  \vskip5pt
\subsection{The grading associated to a nilpotent element $x \in \g_1$}
We fix in this subsection  notation concerning the grading of $\g$ associated to nilpotent elements in $\g_1$ and the correspoding notion of height. The main references for this subsection are \cite{Vi} and \cite{Pa2}. By \cite{Vi}, an element $x \in \gog_1$ is semisimple if and only if $G_0 x$ is closed, whereas it is nilpotent if and only if $0 \in \ol{G_0 x}$. 

Let $x \in \gog_1$ be a nilpotent element. By a modification of the Jacobson-Morozov theorem, there exists a $\mathfrak{sl}(2)$-triple $(x,h,y)$ with $y \in \gog_1$ and $h \in \gog_0$. Such triples are usually called {\it normal triples}, or {\it adapted triples}. Let
$$\gog = \bigoplus_{i \in \mZ} \gog(i)$$
be the $\mathbb Z$-grading defined by $h$; then we get a bigrading of $\g$ by setting
$$ \gog_j(i) = \gog_j \cap \gog(i).$$
Since all normal triples containing $x$ are conjugated by the stabilizer of $x$ in $G_0$ (see \cite[Theorem 1]{Vi}), it follows that the structure of this bigrading does not depend on the choice of the normal triple.

Following Panyushev \cite{Pa2}, define the {\it height} of $x$ as
$$	\height(x) = \max\{n \in \mN \st \gog(n) \neq 0\}. $$
Since $[x,\gog(i)] = \gog(i+2)$, this notion agrees with the height defined in the Introduction, namely the maximum $n$ such that $\mathrm{ad}(x)^n \neq 0$.

%Let $\g$ be simple finite dimensional complex Lie algebra and $\si$ an  automorphism  of $\g$ of order $m$.    Let
%$(\cdot,\cdot)$ be the Killing form of $\g$.  For $=0,1,\ldots,m-1$, set  $\g_j=\{X\in \g\mid \s(X)=\zeta^jX\}$, so that we have  $\g=\g_0\oplus \g_1\oplus\ldots\oplus \g_{m-1}$. 
%Let  $\h_0$ be a Cartan subalgebra of $\g_0$. 

  \vskip5pt
\subsection{Twisted loop algebra and finite order automorphisms}\label{ta}

 Since $\s$ fixes $h_{reg}$,  we see that the action of $\s$ on the positive roots defines, once  Chevalley generators are fixed, a dia\-gram automorphism $\eta$ of $\g$ that, clearly, fixes $\h_0$. Set, using the notation of \cite{Kac}, $\ha=\h_0\oplus\C K\oplus\C d$. Recall that $d$ is the element of 
 $$\widehat L(\g,\s)=(\C[t,t^{-1}]\otimes \g)\oplus\C K\oplus\C d$$ acting on $\C[t,t^{-1}]\otimes \g$ as $t\frac{d}{dt}$, while $K$ is a central element.
Define 
$\d'\in\ha^*$  by setting $\d'(d)=1$ and $\d'(\h_0)=\d'(K)=0$ and let $\l\mapsto \ov\l$ be the restriction map $\ha\to\h_0$. 
There is a unique extension, still denoted by $(\cdot,\cdot)$,  of the 
 Killing form of $\g$ to a nondegenerate symmetric bilinear invariant form on 
 $\widehat L(\g,\si)$.  Let $\nu:\ha\to\ha^*$ be 
the isomorphism induced by the form $(\cdot,\cdot)$, and denote again by $(\cdot,\cdot)$
the form induced on $\ha^*$. One has  $(\d',\d')=(\d',\h_0^*)=0$.

We let $\Da$ be the set of $\ha$-roots of $\widehat L(\g,\s)$. We can choose as set of  positive roots $\Dap=\Dp_{0}\cup\{\a\in\Da\mid \a(d)>0\}$. We let $\Pia=\{\a_0,\dots,\a_n\}$ be the corresponding set of simple roots. It is known that $n$ is the rank of $\g_0$.  
Recall that any $\widehat L(\g,\s)$ is a Kac-Moody Lie algebra $\g(A)$ defined by generators and relations starting from a generalized Cartan matrix $A$ of affine type. These matrices are
classified by means of Dynkin diagrams listed in \cite{Kac}.

Let   $\Wa$ be the Weyl group of $\widehat L(\g,\s)$ and let $\Da_{re}=\widehat W\Pia$ be the set of real roots of $\widehat L(\g,\s)$. Recall that if $\beta=w(\a),\,\a\in\Pia$, one defines
$\beta^\vee=w(\a^\vee)$.

If $\gamma\in\h_0^*$, we set 
$$h_\gamma=\nu^{-1}(\gamma),\qquad\gamma^\vee=\frac{2 h_\gamma}{(\gamma,\gamma)}\quad (\gamma\ne 0).$$
By \cite[5.1]{Kac} if $\l\in \C\d'+\h_0^*$ and $\be\in\Da_{re}$, then 
$$
\l(\beta^\vee)=2\frac{(\l,\be)}{(\be,\be)}=2\frac{(\l,\ov\be)}{(\ov\be,\ov\be)}=\l(\ov\beta^\vee)=2\frac{(\ov\l,\ov\be)}{(\ov\be,\ov\be)}=\ov\l(\ov\beta^\vee).
$$
If $\l,\mu\in \C\d'+\h_0^*$ and $(\mu,\mu)\ne0$, we set
$$
\langle \l,\mu^\vee\rangle=2\frac{(\l,\mu)}{(\mu,\mu)}.
$$
In particular, if $\a\in\Da$ and $\be\in\Da_{re}$,
$$
\langle \a,\be^\vee\rangle=\a(\be^\vee)=\langle\ov\a,\ov\be^\vee\rangle=\ov\a(\ov\be^\vee ).
$$
 We will use these equalities many times without comment.

Following \cite[Chapter 8]{Kac}, we can assume that $\s$ is the automorphism of type $(\eta;s_0,\dots, s_n)$, where $\eta$ is the diagram automorphism defined above. 
%Note that, since $\s$ is an involution, $\eta^2=Id$. We do not assume here that $\g$ is simple, but, as explained in \cite{MMJ}, most arguments given in \cite{Kac} can be safely extended to the  setting where $\g$ is semisimple but not simple. This latter case, i.e. $\g=\k\oplus\k,\,\k$ a simple Lie algebra, $\s$ the flip, will be referred to as the {\it adjoint case}.  
Recall that, if $a_0,\dots,a_n$ are the labels of the Dynkin diagram of $\widehat L(\g,\s)$ and $k$ is the order of $\eta$, then 
$k(\sum_{i=0}^ns_ia_i)=m$. Recall also that $s_0,\dots,s_n$ are relatively prime so, in the case of involutions ($m=2$), we must have that $s_i\in\{0,1\}$ and $s_i= 0$ for all but at most two indices.\par
Since $\s$ is the automorphism of type $(\eta;s_0,\dots, s_n)$, we can write $\a_i=s_i\d'+\ov{\a_i}$ and it turns out that $\Pi_0=\{\a_i\mid s_i=0\}$.
Set also $\Pi_1=\Pia\setminus \Pi_0$.
Introduce $\d=\sum_{i=0}^na_i\a_i$ and note that $\d=(\sum_{i=0}^na_i s_i)\d'=\frac{m}{k}\d'$.
%Set also $\a_i^\vee=\frac{2}{(\a_i,\a_i)}\nu^{-1}(\a_i)$ and let $\{a^\vee_0,\dots,a^\vee_n\}$ be the labels of the dual Dynkin diagram of $\widehat L(\g,\s)$. 
%We assume that $K$ is the canonical central element \cite [6.2]{Kac},  $K=\sum_{i=0}^n a_i^\vee \a_i^\vee$. If we number the Dynkin diagrams as in \cite[Tables Aff1, Aff2, Aff 3]{Kac} then, by Sections  6.1, 6.2, 6.4 of \cite{Kac}, 
%\begin{equation}\label{C}
%K=\frac{2a_0}{\Vert \d-a_0\a_0\Vert}\nu^{-1}{(\d)}.
%\end{equation}
%Set finally  ${\bf g}=\sum_{i=0}^n a_i^\vee$. This number is called the dual Coxeter number of $\widehat L(\g,\s)$.
%We let $\Wa$ be the Weyl group of $\widehat L(\g,\s)$.
% Set $(\h_0)_\R=\oplus_{\a\in\Pi}\R\a^\vee$ and $\ha_\R=\R d\oplus \R K\oplus(\h_0)_\R$. Denote by
%$\label{caf}
%C_1=\{h\in(\h_0)_\R\mid \ov\a_i(h)\geq -s_i,\,i=0,\ldots,n\}
%$
 %the fundamental alcove of $\Wa$. \par
 
Given $\grl \in \widehat \Phi$ we denote by $\widehat L(\g,\s)_\grl$ the corresponding root space in $\widehat L(\g,\s)$. Recall the following properties (see {\cite[Exercise 8.2]{Kac}}).

\begin{proposition} \label{prop:kac}
Let $\grl \in \Da_{re}$, then the following holds:
\begin{itemize}
	\item[i)] $\dim \widehat L(\g,\s)_\grl = 1$.
	\item[ii)] If $\mu \in\Da$, then the set of $\mu + i \grl \in\Da \cup \{0\}$ is a string $\mu - p \grl \ldots, \mu + q \grl$, where $p, q$ are non-negative integers such that $p-q = \langle\mu,\l^\vee\rangle$.
	\item[iii)] If $\mu \in\Da$ and $\mu + \grl \in\Da$, then $[\widehat L(\g,\s)_\grl, \widehat L(\g,\s)_\mu] \neq 0$.
\end{itemize}
\end{proposition}
Notice that, if $\grl = i \grd' + \gra\in\Da_{re}$, then 
$$\widehat L(\g,\s)_\grl =t^i\otimes\g^\a_i.$$
This implies that  we can rephrase the previous proposition in terms of $\h_0$-weights in $\gog$ as follows.

\begin{corollary}	\label{cor:commutatori}
Let $\gra \in \Phi_i$ and $\grb \in \Phi_j$. If $i \equiv j \ mod\ m$, assume also that $\gra \neq \grb$.
\begin{itemize}
	\item[i)] $\dim \gog_i^\gra = 1$, and if $-\gra \in \Phi_i$ then $[\gog_i^\gra , \gog_i^{-\gra}] \neq 0$.
	\item[ii)] If $(\gra, \grb) < 0$ then $\gra + \grb \in \Phi_{i+j}$, and if $(\gra, \grb) > 0$ then $\gra - \grb \in \Phi_{i-j}$.
	\item[iii)] If $\gra + \grb \in \Phi_{i+j}$, then $[\gog_i^\gra , \gog_j^\grb] \neq 0$.
\end{itemize}
\end{corollary}

%\begin{proof}
%Let $\grl = i \grd + \gra$ and $\mu = j\grd + \grb$, then $\grl, \mu \in\Da_{re}$, then $\ol \grl = \gra$ and $\ol \mu = \grb$.

%i) Being $\dim \gog_i^\gra = \dim \widehat L(\g,\s)_\grl$, the first claim follows by Proposition \ref{prop:kac} i). Set $\grl' = i \grd - \gra$, then $\grl' \in\Da_{re}$ and $\grl + \grl' = 2i\grd \in\Da$. Therefore by Proposition \ref{prop:kac} iii) we get
%$$
 %t^{2i} \otimes [\gog_i^\gra, \gog_i^{-\gra}] = [t^i \otimes \gog_i^\gra, t^i \otimes \gog_i^{-\gra}] = [\widehat L(\g,\s)_\grl, \widehat L(\g,\s)_{\grl'}] \neq 0.
%$$

%ii) The claim follows immediately applying Proposition \ref{prop:kac} ii) to $\grl$ and $\mu$.

%iii) By Proposition \ref{prop:kac} iii) we have
%\[
 %t^{2i} \otimes [\gog_i^\gra, \gog_j^{\grb}] = [t^i \otimes \gog_i^\gra, t^j \otimes \gog_j^\grb] = [\widehat L(\g,\s)_\grl, \widehat L(\g,\s)_\mu] \neq 0. \qedhere
%\]
%$\end{proof}

%Recall that $\Pi_0$ denotes  the set of simple roots of $\g_0$ corresponding to $\Dp_0$.
 In general, $\Pi_0$ is disconnected and we write $\S|\Pi_0$ to mean that $\S$ is a connected component of $\Pi_0$.   Clearly,  the Weyl group $W_0$ of $\g_0$ is the direct product of the $W(\S),\,\S|\Pi_0$.
If $\theta_\S$ is the highest root of $\D(\S)$, set
 \begin{eqnarray*}
\Da_0&=&\{\a+\ganz k\d\mid \a\in\D_0\}\cup\pm\nat k\d,\\
 \Pia_0&=&\Pi_0\cup\{k\d-\theta_\S\mid \S|\Pi_0\},\\ 
 \Dap_0&=&\Dp_0\cup\{\a\in\Da_0\mid \a(d)>0\}.\end{eqnarray*}
 Denote by $\Wa_0$ the  Weyl group of $\Da_0$.
If $\a\in\Da$, let $[\a : \a_i]$  be the coefficient of $\a_i$ in the expansion of $\a$ in terms of $\Pia$. Set 
 $$
 \height_\s(\a)=\sum_{i=0}^n s_i [\a : \a_i]
 $$
 and, for  $i\in\ganz$,   
$$\Da_i=\{\a\in\Da\mid \height_\s(\a)=i\}.$$
Note that if  $\a\in\widehat  \Phi_i$, then  $\ov\a$ is a weight of $\g_i$.

\subsection{$B_0$-stable subalgebras in $\g_1$ and $\s$-minuscule elements}\label{bs} In this subsection  we assume that $\si$ is an (indecomposable) involution. With this assumption, $\Pi_1$ has at most two elements.  If $\Pia$ is simply laced, the real roots of $\Da$ are regarded as long. 
\begin{definition}\label{nc}
We say that $\eta\in\Da$ is complex if $\ov\eta\in\D_0\cap\D_1$.\end{definition}
 It is clear that complex roots can occurr  only if $\rk\, \g_0<\rk\,\g$. Moreover, if  $\g$ is simple and $\rk\,\g_0<\rk\,\g$, then $\eta\in\Da$ is complex if and only if it  belongs to $\Da_{re}$ and it is not long (see \cite{CMPP}). The case of  $\g$ semisimple and not simple  corresponds to $\g$ equal to the sum $\k\oplus\k$  of two isomorphic simple ideals, $\s$ the flip involution, $\g_0=\k$, and $\g_1\simeq\k$ with $\k$ acting on itself via the adjoint representation. 
 %In this case all real roots are complex.
%\footnote{Se non mi sbaglio, da quanto mi dicevate $\a_p$ complessa corrisponde essenzialmente al caso aggiunto. Insomma il vantaggio di questa formulazione con le radici complesse e non complesse \`e che in effetti cos\`i ridimostriamo insieme caso aggiunto e caso simmetrico. Credo che varrebbe la pena osservare qualcosa a riguardo, altrimenti non si capisce bene il perch\'e di questa formulazione diversa da quella di Panyushev.}

For $w\in\Wa$, define its set of inversions
\begin{equation}\label{nw}
N(w)=\{\a\in\Dap\mid w^{-1}(\a)\in-\Dap\}.
\end{equation}
Recall that a finite subset $A$  of positive  roots of an affine root system is of the form $N(w)$ for some $w\in \Wa$ if and only if both $A$ and $\Dap \setminus A$ are closed under root addition (see e.g. \cite{CP2}). We will refer to this property as {\it biconvexity}.\par
   If $\a$ is a real root in $\Dap$, we let $s_\a$ denote the reflection in $\a$. If $\a_i$ is a simple root we set $s_i=s_{\a_i}$.\par

Recall from \cite{CMP} the following

\begin{definition}\label{sm} An element $w\in \Wa$ is called $\s$-minuscule if $N(w)\subset \Da_1$. 
\end{definition}
We denote by $\Wab$ the set of $\s$-minuscule elements of $\Wa$, and  we regard it  as a poset under the weak Bruhat order.

%\vskip5pt
 %\Let $a$ be the squared length of a long root in $\Dap$. Define
%\\begin{equation}\label{pi0*}
%\\Pia_0^*=
%\\Pi_0\cup
%\\left\{ k\d-\theta_\S\mid a\leq2\Vert\theta_\S\Vert^2\right\},
%\\end{equation}
%\\begin{equation}\label{phis}\Phi_\s=
%\\Pia^*_0\cup\{\a+k\d\mid \a\in\Pi_1, \a \text{ {long and noncomplex}}\}
%\\end{equation}
%\\vskip5pt

%\\vskip5pt
%\Consider the set
%\$$
%\D_\s=\bigcup_{w\in W^\s_{ab}}wC_1.
%\$$

%\ If $\a\in\Da$ then we let $H^+_\a=\{h\in(\h_0)_\R\mid \a(d+h)\ge0\}$. 
%\The following result  has been proved in \cite[Proposition 4.1]{IMRN} and refined in \cite{CMPP}.

%\\begin{proposition}\label{Dsigma}
%\$$D_\s=\bigcap_{\a\in\Phi_\s}H^+_\a.
%\$$
%\\end{proposition}

%\\begin{remark} In the adjoint case $\g=\k\oplus\k$, $\k$ simple, $D_\s$ is twice the fundamental alcove of the affine Weyl group of $\k$.
%\\end{remark}

\vskip10pt
We let $\mathcal I_{ab}^\s$ be the set of abelian subalgebras of $\g$ contained in $\g_1$ that are stable under the action of the Borel subalgebra $\b_0$ of $\g_0$ corresponding to $\Dp_0$, or equivalently under the action of the Borel subgroup $B_0 \subset G_0$ with Lie algebra $\b_0$. Inclusion turns $\mathcal I_{ab}^\s$ into a poset. 

\begin{proposition}\cite[Theorem 3.2]{CMP}\label{imrn}  Let $w\in\Wab$. Suppose $N(w)=\{\beta_1,\ldots,\beta_k\}$. The map $\Theta:\Wab\to\mathcal I_{ab}^\s$ defined by
$$w\mapsto\bigoplus_{i=1}^k \g_1^{-\ov\beta_i}$$
is a poset isomophism.
\end{proposition}

%\begin{remark}\label{inL}
%The natural isomorphism of $\g^{\ov 0}$-modules $\g^{\ov 1}\cong t^{-1}\otimes \g^{\ov 1}$ { maps the $\b^{\ov 0}$-stable abelian subspaces of $\g^{\ov 1}$ to  $\b^{\ov 0}$-stable abelian subspaces of $\widehat L(\g,\s)$.} Through this isomorphism, the map of the above proposition associates to $w\in \Wab$ the $\b^{\ov 0}$-stable abelian  subalgebra 
%$\bigoplus_{i=1}^k\widehat L(\g,\s)_{-\be_i}$.
%\end{remark}

Assume that $\g_0$ is semisimple; then there is an index $p$ such that  $\Pi_0=\Pia\setminus \{\a_p\}$. Assume furthermore that $\a_p$ is non-complex (in particular, $\g$ is simple). Set $\Pi_{0,\a_p}=\Pi_0\cap\a_p^\perp$, $W_{0, \a_p}=W(\Pi_{0,\a_p}),$
and denote by $w_{0, \a_p}$ the longest element of $W_{0, \a_p}$.  Let  $w_0$ be the longest element of  $W_0$. Set
%\begin{equation}\label{wp}
$$w_{p}=s_p w_{0, \a_p}w_0$$
%\end{equation}
In \cite{CMPP} it is proved that
$w_{p}\in \Wab$ if and only if $\a_p$ is long; in such a case, the abelian $B_0$-stable subalgebra $\aa_p$ corresponding to $w_{p}$ is maximal. 

%%%%%%%%%%%%%%%%%%%%%%%%%%%%%%%%%%%%%%%%%%%%%%%%%%%%%%%
%%%%%%%%%%%%%%%%%%%%%%%%%%%%%%%%%%%%%%%%%%%%%%%%%%%%%%%
\section{$B_0$-orbits in $B_0$-stable subalgebras contained in $\gog_1$}\label{s4}
%%%%%%%%%%%%%%%%%%%%%%%%%%%%%%%%%%%%%%%%%%%%%%%%%%%%%%%
%%%%%%%%%%%%%%%%%%%%%%%%%%%%%%%%%%%%%%%%%%%%%%%%%%%%%%%

Throughout this section, we will assume that $\grs : \gog \lra \gog$ is an (indecomposable) automorphism of order $m$, and that $\aa$ is a $B_0$-stable subalgebra of $\g$ contained in $\gog_1$. By \cite[Proposition 4.9]{Pa3}, $\aa$ contains no semisimple element. In particular, $\aa \cap \gog^0_1 = 0$ and $\aa$ is completely determined by its set of weights $\Psi(\aa) \subset \Phi_1$:
$$
	\aa = \bigoplus_{\gra \in \Psi(\aa)} \gog_1^\gra.
$$

Since $\aa$ is $B_0$-stable it follows that $G_0 \aa$ is closed; since it contains no semisimple element it follows that every element in $\aa$ is nilpotent. Since there are only finitely many nilpotent $G_0$-orbits in $\gog_1$, it follows that $G_0 \aa$ is the closure of such an orbit.

For all $\gra \in \Phi_i$, fix a non-zero element $x_i^\gra \in \gog_i^\gra$. If $v \in \aa$ and $v = \sum_\a c_\gra x_1^\gra$, then we set  $\supp(v) = \{\gra \in \Psi(\aa) \st c_\gra \neq 0\}$. If $\calS \subset \Psi(\aa)$ we set 
$$x_\calS = \sum_{\gra \in \calS} x_1^\gra.$$

\begin{theorem} \label{teo:B0-orbite}
Let $\aa$ be a $B_0$-stable abelian subalgebra in $\g_1$. For all $x\in\aa$, there is a unique orthogonal subset $\calS$ of $\Psi(\aa)$ such that $B_0x = B_0x_{\calS}$.
In particular, $B_0$ acts on $\aa$ with finitely many orbits, which are parametrized by the orthogonal subsets of $\Psi(\aa)$.
\end{theorem}

 In the special case of the involution $\sigma:\g\oplus \g\to \g\oplus\g$, $(x,y)\mapsto(y,x)$, $\mathcal I^\s_{ab}$ is the set of abelian ideals of $\gob$ and, for such an ideal $\aa$, $G\aa$ is always the closure of a spherical nilpotent $G$-orbit in $\gog$ \cite{PR}. Moreover, the closure of a spherical nilpotent orbit in $\gog$  can be realized as $G\aa$ for some abelian ideal  $\aa$ (up to choosing  the Borel subalgebra $\gob$ in a compatible way). In this case, Panyushev \cite{Pa5} has recently given a new proof of the finiteness of the $B$-orbits, by giving an explicit parametrization of the $B$-orbits in $\aa$. Our proof of Theorem \ref{teo:B0-orbite} will follow closely the proof of \cite[Theorem~2.2]{Pa5}.

By Corollary \ref{cor:commutatori}, the following properties for a  $B_0$-stable abelian subalgebra  $\aa\subset\gog_1$ hold.
\begin{itemize}
	\item[(A1)] If $\gra \in \Psi(\aa)$, then $-\gra \not \in \Psi(\aa)$.
	\item[(A2)] Let $\gra, \grb \in \Psi(\aa)$, then $\gra+\grb \not \in \Phi_2$.
	\item[(A3)] Let $\gra \in \Psi(\aa)$ and $\grg \in \Phi_0^+$ be such that $\gra + \grg \in \Phi_1$, then $\gra + \grg \in \Psi(\aa)$.
\end{itemize}

\begin{lemma} 	\label{prop:orthogonality}
For  $\gra, \grb \in \Psi(\aa),\,\gra\ne\grb,$ the following statements are equivalent:
\begin{itemize}
\item[i)]  $\gra, \grb$ are orthogonal;
\item[ii)] $\gra - \grb \not \in \Phi_0$;
\item[iii)] $\grd'+\gra$ and $\grd'+\grb$ are strongly orthogonal in $\Da$.
\end{itemize}
\end{lemma}

\begin{proof} 
$(1)\Rightarrow (2)$ Suppose that $(\gra,\grb) = 0$ and $\gra - \grb \in \Phi_0$. Then $\grd'+\gra, \grd' + \grb \in\Da$ are orthogonal as well, and
$$
	s_{\grd'+\grb}(\gra - \grb) = s_{\grd'+\grb}(\grd'+\gra)- s_{\grd'+\grb}(\grd' + \grb) = 2 \grd '+ \gra + \grb.
$$
It follows that $2 \grd '+ \gra + \grb \in\Da$, that is $\gra + \grb \in \Phi_2$, contradicting (A2).

$(2)\Rightarrow (1)$
Suppose that $(\gra,\grb) \neq 0$. If $(\gra, \grb) < 0$, then Corollary \ref{cor:commutatori} implies $\gra + \grb \in \Phi_2$, contradicting (A2). Therefore it must be $(\gra, \grb) > 0$, and again by Corollary \ref{cor:commutatori} we get $\gra - \grb \in \Phi_0$.
\par
Statement  (3) is clearly equivalent to the  others.
\end{proof}

As in \cite{Pa5}, the key step to prove Theorem~\ref{teo:B0-orbite} is the following combinatorial lemma, which generalizes \cite[Lemma~1.2]{Pa5} to the graded setting.
% by making use of the affine root system $\widehat \Phi$.

\begin{lemma}	\label{lem:adding-roots}
Let $\gra, \grb \in \Psi(\aa)$ be orthogonal weights and let $\grg \in \Phi_0$.  If $\gra + \grg \in \Psi(\aa)$, then $\grb + \grg \not \in \Psi(\aa)$
%\begin{itemize}
%	\item[i)] If $\gra + \grg \in \Psi(\aa)$, then $\grb + \grg \not \in \Psi(\aa)$.
%	\item[ii)] If $\gra - \grg \in \Psi(\aa)$, then $\grb - \grg \not \in \Psi(\aa)$.
%\end{itemize}
\end{lemma}

\begin{proof}
Assume that both $\gra + \grg$ and $\grb + \grg$ belong to $\Psi(\aa)$. Suppose that $(\gra,\grg) < 0$: then $(\grb + \grg,\gra) < 0$ as well, and Corollary \ref{cor:commutatori} implies $\gra + \grb + \grg \in \Phi_2$, against (A2). Similarly it cannot be $(\grb,\grg) < 0$. Suppose that $(\gra, \grg) = (\grb, \grg) = 0$: then $(\gra+\grg, \grb+\grg) > 0$, hence Corollary~\ref{cor:commutatori} ii) implies $\gra - \grb \in \Phi_0$, which contradicts the fact that $\gra$ and $\grb$ are orthogonal by Lemma~\ref{prop:orthogonality}. Therefore it must be $(\gra,\grg) \geq 0$ and $(\grb,\grg) \geq 0$.
\par
On the other hand, again by Lemma~\ref{prop:orthogonality}, we have that 
$$(\gra,\grb)=0\iff \gra-\grb\notin \Phi_0 \iff (\gra+\grg)-(\grb+\grg)\notin \Phi_0 \iff (\gra+\grg,\grb+\grg)=0.$$
In turn, by the orthogonality of $\gra$ and $\grb$, the last equality implies that either  $(\gra,\grg)<0$ or $(\grb,\grg)<0$, which is a contradiction.
\end{proof}

%Suppose that $\theta$ is an inner automorphism. Then $\rk \gog_0 = \rk \gog$, hence $\Phi = \bigcup \Phi_i$ and it follows that $\gog$ is of type $G_2$.
%
%\underline{Da qui in poi assumo che $\theta$ sia un'involuzione.} Suppose that $\theta$ is an inner involution, then there is a unique $\mZ_2$-grading of $\gog$, corresponding to the simple root $\gra_2$: it follows that $\mathfrak b_0 = \gou_{\gra_1} \oplus \gou_\rho$ (where $\rho$ denotes the highest root), and $\gog_0$ is a subalgebra of $\gog$ of type $A_1 \times A_1$. Since there are no $\rho$-strings of length greater than 2, the unique possibile case is $\grg = \gra_1$, $\gra = 2\gra_1 + \gra_2$ and $\grb = -\gra_1 -\gra_2$, which is absurd since then $(\gra,\grb) \neq 0$. This shows that it cannot be $(\gra,\grg) > 0$, and similarly it cannot be $(\grb,\grg) > 0$.
%
%Suppose that $\theta$ is an outer involution. Since $\gog_0$ is never of type $G_2$, it follows that $\Phi_1 \not \subset \Phi_0$. By direct check we see that $\Phi_1$ never contains such a string in all cases (that is, cases (S8), (S9) and (S10) in \cite{Pa4}), and we get a contradiction.
%
%ii) Suppose that both $\gra - \grg$ and $\grb - \grg$ belong to $\Psi(\aa)$. Then by Proposition \ref{prop:orthogonality} we get that $(\gra - \grg) - (\grb - \grg) = \gra - \grb \not \in \Phi_0$. Therefore by Proposition \ref{prop:orthogonality} again it follows that $\gra - \grg$ and $\grb - \grg$ are orthogonal, and we get an absurd by i).

We denote by $\leq_0$ the dominance order on $\h_0^*$: $\grl \leq_0 \mu$ if $\mu - \grl \in \mN \Phi^+_0$. 
%Given $\gra \in \Psi(\aa)$ we set
%$$\Psi_\gra(\aa) = \{\grb \in \Psi(\aa) \st \grb - \gra \in \Phi_0^+\}.$$
If $\calS$ is  subset of  $\Psi(\aa)$ we denote by $\min (\calS)$ the set of the minimal elements of $\calS$ w.r.t. $\leq_0$, and define two subsets of $\Psi(\aa)$ as follows
\begin{eqnarray*}
\Psi_\calS &=& \{\grb \in \Phi_1 \st \mbox{ there is } \gra \in \calS \mbox{  with } \grb - \gra \in \Phi_0^+\}, \\
\calS^{{\geq}_0} &=& \{\grb \in \Phi_1 \st \mbox{\ there is\ } \gra \in \calS \mbox{ with } \gra \leq_0 \grb \}.
\end{eqnarray*}
We also denote by $\aa_\calS$ the minimal $B_0$-stable subalgebra of $\aa$ containing the weight space $\gog_1^\gra$ for all $\gra \in \calS$, namely
$$
	\aa_\calS = \bigoplus_{\grb \in \calS^{{{\geq}_0}}} \gog_1^\grb.
$$

\begin{lemma}	\label{lemma:argomento-panyushev}Let $\aa\in\mathcal I^\s_{ab}$. Let $\calS$ be an orthogonal subset of $\Psi(\aa)$. Let $x \in \aa$ be such that $\calS\subset\supp(x)$ and $\calS$ is  a lower order ideal in  $\supp(x)$. Then there is $y \in B_0 x$ with the same property such that $\supp(y) \cap \Psi_\calS = \vuoto$.
\end{lemma}

\begin{proof}Set $Z=\{x\in\aa\mid \calS \mbox{  is a lower order ideal in  } \supp(x)\}$. If $x\in Z$, we set $\mathcal T(x)=\supp(x)\setminus \calS$. We prove the claim by induction on $\dim\aa_{\mathcal T(x)}$. If $\dim\aa_{\mathcal T(x)}=0$, then $\supp(x)=\calS$ and there is nothing to prove.

Assume $\dim\aa_{\mathcal T(x)}>0$. 
Set $\calS' = \Psi_\calS$, and for $v \in \aa$, let $\calS'(v) = \calS' \cap \supp(v)$. We can assume that $\calS'(x) \neq \vuoto$, for, otherwise, we can take $y=x$. Then there are $\gra \in \calS$ and $\grg \in \Phi_0^+$ such that $ \gra + \grg\in\supp(x)$. By Lemma~\ref{lem:adding-roots}, it follows that $\gre + \grg \not \in \Psi(\aa)$ for all $\gre \in \calS \setminus \{\gra\}$. If $u_\grg(\xi) \in B_0$ is the element defined by exponentiating $\xi x^\grg_0$ ($\xi \in \mathbb C$), it follows that 
\begin{equation}\label{exp}u_\grg(\xi) x^\be_1= x^\be_1\mbox{ if }\be\in \calS\setminus\{\a\},\quad
u_\grg(\xi) x^\be_1= x^\be_1 + \xi [x^\grg_0,x^\be_1] + \frac{\xi^2}{2} [x^\grg_0,[x^\grg_0,x^\be_1]] + \ldots  \mbox{ otherwise.}
\end{equation}
Let $\pi : \aa \lra \bigoplus_{\gra \in \calS} \gog_1^\gra$ be the projection. We claim that $\pi(u_\grg(\xi) x) = \pi(x)$ for all $\xi \in \mathbb C$. In fact, if $\varepsilon\in\calS\setminus\{\a\}$ then, by (\ref{exp}), $\pi(u_\grg(\xi) x^\gre_1)=x^\gre_1$. Since $\calS$ is strongly orthogonal, $\a+k\gamma\notin \calS$ so $\pi(u_\grg(\xi) x^\gra_1)=x^\gra_1$. Finally, if $\be\notin \calS$ and $\be\in\supp(x)$, then, since $\calS$ is a lower order ideal in $\supp(x)$,  $\be+k\gamma\notin\calS$, so $\pi(u_\grg(\xi) x^\be_1)=\pi(x^\be_1)=0$. 
Choose $\xi_0 \in \mathbb C$ such that $\a+\grg \not \in \supp (u_\grg(\xi_0) x)$ and set
 $x' = u_\grg(\xi_0) x$. By construction $\calS \subset \supp(x')$. We have to prove that  $\calS$ is a lower order ideal in $\supp(x')$. Take $\eta\in\supp(x')$ such that there exists $\beta\in \calS$ with 
 $\eta\leq \beta$. We know that $\eta=\zeta + r \gamma,\, \zeta\in \supp(x)$. Since $\zeta\leq_0 \beta$, we have that $\zeta\in\calS$.
 If $\zeta\ne \a$, then, by Lemma \ref{lem:adding-roots}, $\zeta+\gamma$ is not in $\Phi_1$, in particular $\eta\notin \supp(x')$. If $\zeta=\a$ then, by construction, either $r>1$ or $\eta=\zeta$. If $r>1$ then $\a+\gamma\le_0\a+r\gamma\le_0 \be\in\calS$ and $\a+\gamma\in\supp(x)$, so $\a+\gamma\in\calS$. We already observed that this is not possible, hence  $\eta=\zeta\in\calS$.

   We now prove  that $\aa_{\mathcal T(x')}\subset\aa_{\mathcal T(x)}$.  It suffices to prove that, if $\be\in \mathcal T(x')$ then  $\be\in\Psi(\aa_{\mathcal T(x)})$. Write $\be=\zeta+r\gamma$ with $r\ge0$ and $\zeta\in\supp(x)$. If $\zeta\notin\calS$ then we are already done. If $\zeta\in\calS$ then, as shown above, we have $\zeta=\a$. Note that $r>0$, for, otherwise $\be\in\calS$. Now, if $r>0$, as $\a+\gamma\in\mathcal T(x)$ and $\a+\gamma\le_0\be$, we have that $\be\in\Psi(\aa_{\mathcal T(x)})$.

Since $\a+\gamma \not \in \supp(x')$, it follows that $\aa_{\mathcal T(x')} \subsetneq \aa_{\mathcal T(x)}$. By the induction hypothesis,  it follows that $B_0 x'$ contains an element $y$ such that $ \calS$ is a lower order ideal in  $\supp(y)$ and $\supp(y) \cap \Psi_\calS = \vuoto$. Since $x'\in B_0x$, we have that $y\in B_0x$.
\end{proof}
Thanks to previous lemmas, we can now reproduce the same argument given in \cite[Theorem~2.2]{Pa5} to prove  Theorem~\ref{teo:B0-orbite}.

%We can now prove the main theorem of this section.

\begin{proof}[Proof of Theorem~\ref{teo:B0-orbite}]
We first show that every $B_0$-orbit in $\aa$ possesses a representative of the form $x_\calS$, for some orthogonal subset $\calS \subset \Psi(\aa)$.

Let $v \in \aa$. Set $v_0 = v$ and $\calS_0 = \min \supp(v)$. Notice that $\calS_0$ is ortohogonal by Lemma~\ref{prop:orthogonality}: indeed for all $\gra, \grb \in \calS_0$ we have  by construction that $\gra - \grb \not \in \Phi_0$. Therefore $v_0$ and $\calS_0$ satisfy the assumptions of Lemma~\ref{lemma:argomento-panyushev}, and there is $v_1 \in B_0 v_0$ such that $ \calS_0 \subset \supp(v_1)$ and $\supp(v_1) \cap \Psi_{\calS_0} = \vuoto$. Define $\calS_1 = \calS_0 \cup \min (\supp(v_1) \setminus \calS_0)$: then by construction we still have $\gra - \grb \not \in \Phi_0$ for all $\gra, \grb \in \calS_1$, so that $\calS_1$ is again orthogonal by Lemma~\ref{prop:orthogonality}. It is clear that $\calS_1$ is lower order ideal in $\supp(v_1)$.

More generally, let $i \geq 0$ and suppose that $v_i$ and $\calS_i$ are defined. Then there is $v_{i+1} \in B_0 v_i = B_0 v_0$ such that $\calS_i$ is a lower order ideal in $\supp(v_{i+1})$, $\supp(v_{i+1}) \cap \Psi_{\calS_i} = \vuoto$, and
$$
	\calS_{i+1} = \calS_i \cup \min (\supp(v_{i+1}) \setminus \calS_i)
$$
is an orthogonal subset by Lemma~\ref{prop:orthogonality} that is a lower order ideal in $\supp(v_{i+1})$.

Clearly, $\calS_{i+1}$ is strictly bigger than $\calS_i$, unless $\supp(v_{i+1}) = \calS_i$. Therefore, proceeding inductively, we find an element $v_{k+1} \in B_0v$ whose support equals $\calS_k$, which is an orthogonal subset, hence $B_0v=B_0v_{k+1}=B_0x_{\calS_k}$.

We now show that every $B_0$-orbit contains a unique orthogonal representative $x_\calS$. If $\mathcal S \subset \Psi(\aa)$, notice that the vector space $\langle B_0 x_\calS \rangle$ generated the orbit of $x_\calS$ is $B_0$-stable, therefore it coincides with $\aa_\calS$.

Let $\mathcal S, \mathcal S'$ be orthogonal subsets of $\Psi(\aa)$ and suppose that $B_0 x_\calS = B_0 x_{\mathcal S'}$. Then $\aa_\calS = \aa_{\calS'}$, and we set $\grG = \min \Psi(\aa_\calS)$. Notice that $\grG \subset \calS \cap \calS'$, therefore setting $\calR = \calS \setminus \grG$ and $\calR' = \calS' \setminus \grG$ we can decompose $x_\calS = x_\grG + x_\calR$ and $x_{\calS'} = x_\grG + x_{\calR'}$. Let $B_0=T_0U_0$ be the Levi decomposition of $B_0$. Let $b \in B_0$ be such that $b x_\calS = x_{\calS'}$ and write $b = t^{-1} u$ with $t \in T_0$ and $u \in U_0$. Then $u x_\grG + u x_\calR = t x_\grG + t x_{\calR'}$. It follows that $t x_\grG = u x_\grG = x_\grG$, hence $x_{\calR'} \in B_0 x_\calR$. Since $\aa_\calR \subset \aa_\calS \cap \aa_{\calS'}$ is a smaller $B_0$-stable abelian subalgebra in $\gog_1$ and since $\calR$ and $\calR'$ are orthogonal subsets in $\Psi(\aa_\calR)$, the claim follows proceeding by downward induction.
\end{proof}

The following facts are also proved by adapting the same proofs of \cite{Pa5}. Denote by $\calS_\aa \subset \Psi(\aa)$ the subset constructed as follows: set $\calS_1 = \min \Psi(\aa)$, and for $i > 1$ define inductively
$$
	\calS_i = \min\big(\Psi(\aa) \setminus \bigcup_{j < i} (\calS_j \cup \Psi_{\calS_j}) \big).
$$
Define $\calS_\aa = \bigcup_{i > 0} \calS_i$, which is an orthogonal subset  thanks to Lemma \ref{prop:orthogonality}.

\begin{proposition}\label{prop:varia-panyushev}
Let $\calS \subset \Psi(\aa)$ be an orthogonal subset.
\begin{itemize}
	\item[i)] $B_0 x_\calS$ is open in $\aa$ if and only if $\calS = \calS_\aa$.
	\item[ii)] As a $T_0$-module, the tangent space $\mathrm T_{x_\calS} (B_0 x_\calS)$ decomposes as follows:
$$
\mathrm T_{x_\calS} (B_0 x_\calS) = \bigoplus_{\gra \in \calS \cup \Psi_\calS} \gog_1^\gra.
$$
In particular, $\dim B_0 x_\calS = |\calS| + |\Psi_\calS|$.
\end{itemize}
\end{proposition}

\section{Antichains of orthogonal roots in Hermitian symmetric spaces}\label{4}
%%%%%%%%%%%%%%%%%%%%%%%%%%%%%%%%%%%%%%%%%%%%%%%%%%%%%%%
%%%%%%%%%%%%%%%%%%%%%%%%%%%%%%%%%%%%%%%%%%%%%%%%%%%%%%%

Suppose that $\gog$ is a simple Lie algebra and let $\Pi$ be a set of simple roots. Let $\theta$ be the corresponding highest root and let $\gra_q \in \Pi$ be a simple root with $[\theta : \gra_q] = 1$.
%\footnote{A me sembra un po' fuorviante la notazione $\a_q$, perch\'e non \`e determinata da $\Pi$. Capisco che \`e per conformit\`a con quello che facciamo con $(\grS, \a_\grS)$, per\`o in quel caso $\a_\Sigma$ \`e effettivamente determinata da $\Sigma$ (avendo fissato $\a_p$). Comunque se pensate che sia pi\`u chiaro cos\`i non c'\`e problema.} 
Let $\gop^+ \subset \gog$ be the maximal parabolic subalgebra  associated to the set of simple roots $ \Pi \setminus \{\gra_q\}$. Then its nilradical  $\gop^+_\mru$ is abelian; conversely any standard parabolic subalgebra with abelian  nilradical arises in this way.

Let $\gop^+ = \mathfrak l\oplus\gop^+_\mru$ be the Levi decomposition, and let $\gop^-$ be the opposite parabolic subalgebra of $\gop^+$. The decomposition into $\mathfrak l$-submodules $\gog = \mathfrak l \oplus \gop_\mru^+ \oplus \gop_\mru^-$ defines an involution $\grs$ of $\gog$ by setting $\grs(x) = x$ if $x \in \mathfrak l$ and $\grs(x) = -x$ if $x \in  \gop_\mru^+ \oplus \gop_\mru^-$. It is then clear that
$$
\g_0=\mathfrak l,\quad \g_1=\gop_\mru^+ \oplus \gop_\mru^-.$$

Recalling the notation introduced in Section \ref{ta}, the sets $\Phi_i$ attached to $\s$ are $\Phi_1=\Phi_1^+\cup-\Phi_1^+$, where  $\Phi_1^+$ is the set of roots $\be$ such  $[\be:\a_q]=1$, while $\Phi_0$ is the set of roots $\be$ such that $[\be:\a_q]=0$. Observe that in this case $\Phi=\Phi_0\cup\Phi_1$ is the set of roots of $\g$. Since in this case $\h_0$ is a Cartan subalgebra of $\g$, we can choose $h_{reg}\in\h_0$ so that $\a(h_{reg})>0$ for all $\a\in\Pi$.  With this choice, letting $\Phi^+$ denote the set of positive roots of $\g$ corresponding to the choice of $\Pi$, we have that
\begin{eqnarray*}
	\Phi_0^+ &=& \{\grb \in \Phi^+ \st [\grb:\gra_q] = 0\},\\
	\Phi_1^+ &=& \{\grb \in \Phi^+ \st [\grb:\gra_q] > 0\},\\
	\Pi_0 &=&\Pi\setminus\{\a_q\}.
\end{eqnarray*}
We set $\Phi_i^- = -\Phi_i^+$ ($i=0,1$). Clearly, $ \Phi_1^\pm$ is the set of weights of $\h_0$ in $\gop_\mru^\pm$.  We let $\b_0^\pm$ be the Borel subalgebra of $\g_0$ corresponding to $\Phi_0^\pm$.  Recall that $W_0$ denotes the Weyl group of $\g_0$. 

Let $\Ort(\Phi_1^+)$ be the collection of the orthogonal subsets of $\Phi_1^+$, and let $\Ort_{\max}(\Phi_1^+)$ be the collection of the orthogonal subsets of $\Phi_1^+$ which are maximal with respect to  inclusion. Regard $\Phi_1^+$ as a partial ordered set via  $\leq_0$. Since $\gop^+_\mru$ is an abelian subalgebra of $\gog$, by Lemma~\ref{prop:orthogonality} two elements $\gra,\grb \in \Phi_1^+$ are orthogonal if and only if they are strongly orthogonal, if and only if $\gra - \grb \not \in \Phi_0$. In particular, every antichain $\calA \subset \Phi_1^+$ is an orthogonal subset.

Given $\calB \in \Ort(\Phi_1^+)$, let $\mathfrak a_\calB \subset \mathfrak p^+_\mru$ be the $B_0$-stable subalgebra generated by $\calB$. We define a preorder $\vdash$ on $\Ort(\Phi_1^+)$ as follows: if $\calB_1, \calB_2 \in \Ort(\Phi_1^+)$, then $\calB_1 \vdash \calB_2$ if $\mathfrak a_{\calB_1} \subset \mathfrak a_{\calB_2}$. Equivalently, $\calB_1 \vdash \calB_2$ if and only if $\calB_1 \subset \calB_2^{{\geq}_0}$, where, if $\calB \subset \Phi_1^+$, we set
$$
	\calB^{{\geq}_0} = \{ \gra \in \Phi^+_1 \st \mbox{there is }\grb \in \calB \mbox{  such that }  \grb\leq_0\gra\}.
$$

Given $m \leq r$, we will denote by $\Ort_m(\Phi_1^+)$ the set of the orthogonal subsets of cardinality $m$. If $\Phi$ is not simply laced, we will say $\calB \in \Ort(\Phi_1^+)$ \textit{is of type} $(h,k)$ if it contains exactly $h$ short roots and $k$ long roots, and we denote by $\Ort_{(h,k)}(\Phi_1^+)$ the set of the orthogonal subsets of type $(h,k)$. To unify some notations, in the simply laced case we will regard every root as a long root. Therefore if $\Phi$ is simply laced we have $\Ort_m(\Phi_1^+) = \Ort_{(0,m)}(\Phi_1^+)$.

%When $\g$ is simply laced, throughout this section we will consider every root as a short root. If In particular, if $\Phi$ is simply laced and $\calB \subset \Phi_1^+$ is an orthogonal subset of cardinality $m$, then $\calB$ is of type $(m,0)$. We denote by $\Ort_{(h,k)}(\Phi_1^+)$ be the set of the orthogonal subsets of type $(h,k)$.

In this section we will study the antichains of $\Phi_1^+$. We will show that for every orthogonal subset $\calB \subset \Phi_1^+$ there is always an antichain $\calA \subset \Phi_1^+$ such that $\calA\vdash \calB$. We first discuss the simply laced case uniformly; the two remaining cases $(B_n,\gra_1)$ and $(C_n,\gra_n)$ will be treated separately. We summarize our results in the following theorem.

\begin{theorem}\label{antichainbelow}\ 
\begin{itemize}
	\item[i)] Suppose that $\Phi$ is simply laced. For all $\calB \in\Ort(\Phi_1^+)$, there is an antichain $\calA \subset \Phi_1^+$ such that $|\calA| = |\calB|$ and $\calA\vdash \calB$. 
	\item[ii)] Suppose that $\Phi$ is not simply laced. For all $\calB \in\Ort(\Phi_1^+)$ of type $(h,k)$, there is an antichain $\calA \subset \Phi_1^+$ of type $(h+\lfloor k/2 \rfloor, k - 2\lfloor k/2 \rfloor)$ such that $\calA\vdash \calB$. 
\end{itemize}
\end{theorem}

\begin{proof}
The claim follows combining Proposition \ref{prop:anticatene-simply}, Proposition \ref{prop:anticatene-tipoB}, Proposition \ref{prop:anticatene-tipoC}.
\end{proof}

%Let $G$ be the simply connected complex algebraic group with Lie algebra $\gog$, and 
Let $P \subset G$ be the parabolic subgroup corresponding to $\gop^+$. It is well known that $G/P$ is an irreducible simply connected Hermitian symmetric space of compact type, whose corresponding involution of $G$ is $\grs$, and every such a symmetric space arises in this way (see e.g. \cite[Section 5.5]{RRS}). Therefore we will refer to the pair $(\Pi,\a_q)$ as a \textit{Hermitian pair}, and we will say that $\grs$ is an \textit{involution of Hermitian type}, or simply a \textit{Hermitian involution}. Correspondingly, we get also a symmetric variety $G/G_0$, where $G_0 = G^\s$ is the Levi factor of $P$.

Since $\gop_\mru^+ \subset \gog_1$ is a $B_0$-stable abelian subalgebra of $\gog$, by Theorem \ref{teo:B0-orbite} it possesses finitely many $B_0$-orbits, which are classified by $\Ort(\Phi_1^+)$. In this situation, the description of the $B_0$-orbits already follows by \cite{Pa5}: since $\gop_\mru^+$ is abelian, the unipotent radical $P_\mru$ acts trivially on its Lie algebra $\gop_\mru^+$, therefore every $B$-orbit is actually a $B_0$-orbit. The $G_0$-orbits in $\gop_\mru^+$ were studied by Muller, Rubenthaler, and Schiffmann \cite{MRS} and by Richardson, R\"ohrle and Steinberg \cite{RRS}. In the latter reference  it is shown that they are parametrized by the $W_0$-orbits in $\Ort(\Phi^+_{1,\ell})$, where $\Phi^+_{1,\ell} \subset \Phi^+_1$ denotes the subset of the long roots.

Let $\calS_{\Pi,\gra_q} \subset \Phi_1^+$ be the orthogonal subset corresponding to the open $B_0$-orbit of $\gop_\mru^+$, constructed recursively as in Proposition \ref{prop:varia-panyushev}. In this case $\calS_{\Pi,\gra_q}$ is well known, and it coincides with the set of Harish-Chandra strongly orthogonal roots (see \cite{HC}, \cite{M}). Denote by $r = |\calS_{\Pi,\gra_q}|$ the rank of the symmetric variety $G/G_0$. By \cite[Theorem 2]{M} we have $\calS_{\Pi,\gra_q} \subset \Phi^+_{1,\ell}$, in particular $\calS_{\Pi,\gra_q}$ is a maximal orthogonal subset of $\Phi^+_1$ consisting of long roots, and by \cite[Theorem 2.12 and Proposition 2.13]{MRS} it follows that $\calS_{\Pi,\gra_q}$ is an orthogonal subset of maximal cardinality in $\Phi_1^+$. In particular, $\calS_{\Pi,\gra_q} \in \Ort_{\max}(\Phi_1^+)$, and $|\calB| \leq r$ for all $\calB \in \Ort(\Phi_1^+)$.

We report in Table \ref{table:hermitian-pairs} the classification of the Hermitian pairs, together with the rank of the corresponding symmetric varieties $G/G_0$ (where $\Pi = \{\gra_1, \ldots, \gra_n\}$ is enumerated as in \cite{Bou}).
\begin{table}[h]
\begin{center}
\begin{tabular}{|l|c|}
\hline
$\qquad \quad (\Pi,\gra_q)$ & $\rk (G/G_0)$ \\
\hline
\hline
$(A_n,\gra_q)$ ($1 \leq q \leq n$) & $\min\{q,n+1-q\}$\\
$(B_n,\gra_1)$ & 2 \\
$(C_n,\gra_n)$ & $n$ \\
$(D_n,\gra_1)$ & 2 \\
$(D_n,\gra_{n-1})$, $(D_n,\gra_n)$ & $\lfloor \frac{n}{2}\rfloor$ \\
$(E_6,\gra_1)$, $(E_6,\gra_6)$ & 2 \\
$(E_7,\gra_7)$ & 3 \\
\hline
\end{tabular}
\end{center} \caption{Hermitian pairs and ranks of the corresponding symmetric varieties.}\label{table:hermitian-pairs} 
\end{table}

\begin{remark}	\label{oss:max-cardinality}
By \cite[Proposition 2.8 and Remark]{RRS}, the Weyl group $W_0$ acts transitively on $\Ort_{(h,k)}(\Phi_1^+)$ for all $h,k$. In particular, we see that if $\Phi$ is simply laced then $\Ort_{\max}(\Phi_1^+) = \Ort_r(\Phi_1^+)$ coincides with the collection of the orthogonal subsets of maximal cardinality. On the other hand in the non-simply laced cases, corresponding to the Hermitian pairs $(B_n,\gra_1)$ and $(C_n,\gra_n)$, we will easily see that if $\calB$ is an orthogonal subset of type $(h,k)$, then $\calB \in \Ort_{\max}(\Phi_1^+)$ if and only if $2h+k = r$. In particular, it follows that, if $\calB \in \Ort(\Phi_1^+)$ has maximal cardinality, then every root in $\calB$ is long. Hence the orthogonal subsets of maximal cardinality coincide with the elements of $\Ort_{(0,r)}(\Phi_1^+)$, and $W_0$ acts transitively on these subsets. As well, it follows that, by choosing properly the Borel subgroup $B_0 \subset G_0$, every subset of orthogonal roots of maximal cardinality can be made into a set of strongly orthogonal Harish-Chandra roots for $\Phi_1^+$.
\end{remark}

\subsection{The simply laced case.}
We start by recording a well known fact that holds for any root system. See e.g. \cite[Lemma 3.2]{Som}. 
\begin{lemma}	\label{lemma:dominanza}
Let $\grb, \grb' \in \Phi^+$ and suppose that $\grb'-\grb $ is a sum of positive roots. Then there are $\grg_1, \ldots, \grg_m \in \Phi^+$ such that $\grb' - \grb = \grg_1 + \ldots + \grg_m$ and $\grb + \grg_1 + \ldots + \grg_i \in \Phi^+$ for all $i \leq m$.
\end{lemma}

%\begin{proof}
%We proceed by induction on $m$. Suppose that $\langle \grb, \grg_i^\vee \rangle \geq 0$ and $\langle \grb', \grg_i^\vee \rangle \leq 0$ for all $i \leq m$: then
%$$(\grb'-\grb, \grb'-\grb) = \sum_{i=1}^m (\grb'-\grb, \grg_i) \leq 0,$$
%hence $\grb = \grb'$ and $m = 0$. Suppose that $m>1$ and let $i \leq m$ be such that $\langle \grb, \grg_i^\vee \rangle < 0$ (resp. $\langle \grb', \grg_i^\vee \rangle > 0$): then $\grb + \grg_i \in \Phi^+$ and $\grb < \grb + \grg_i < \grb'$ (resp. $\grb' - \grg_i \in \Phi^+$ and $\grb < \grb' - \gra_i < \grb'$), and the claim follows by induction.
%\end{proof}
In the simply laced case, Lemma \ref{lemma:dominanza} can be improved as follows:
\begin{proposition}	\label{prop:dominance-roots}
Suppose that $\Phi$ is simply laced. Let $\grb, \grb' \in \Phi^+$ and suppose that $\grb'-\grb $ is a sum of positive roots. Then $\grb' - \grb$ is a sum of  positive pairwise orthogonal roots.
\end{proposition}

\begin{proof}
By Lemma \ref{lemma:dominanza}, there are $\grg_1,\ldots,\grg_m \in \Phi^+$ such that $\grb' -\grb = \grg_1 + \ldots + \grg_m$ and $\grb + \grg_1 + \ldots + \grg_i$ is a positive root for all $i\leq m$. Let $m$ be minimal with the previous property, fix $\grg_1,\ldots,\grg_m \in \Phi^+$ as above and, if $0 \leq i \leq m$, denote $\grb_i = \grb + \grg_1 + \ldots + \grg_i$. We claim that $\grg_1, \ldots, \grg_m$ are pairwise orthogonal.

If $m=1$ there is nothing to prove. Assume $m>1$, and suppose that $\grg_1, \ldots, \grg_m$ are not orthogonal. Let $i_0 \leq m$ be the minimum such that $(\grg_i,\grg_j)=0$ for all $i,j< i_0$ with $i \neq j$, and let $j_0< i_0$ be such that $(\grg_{j_0},\grg_{i_0}) \neq 0$. Since $\langle \beta,\grg_{i_0}^\vee \rangle \geq -1$ and $\langle \grg_i,\grg_{i_0}^\vee \rangle \geq-1$, we can assume that $\langle \grg_{j_0},\grg_{i_0}^\vee \rangle =-1$, thus $\grg_{j_0}+\grg_{i_0}$ is a root. To reach a contradiction, we show that the $m-1$ positive roots
$$\grg_1, \ldots, \grg_{j_0-1}, \grg_{j_0+1}, \ldots, \grg_{i_0-1}, \grg_{j_0}+\grg_{i_0}, \grg_{i_0+1}, \ldots, \grg_m$$
also satisfy the assumptions of $\grg_1, \ldots, \grg_m$, contradicting the minimality of $m$. That is, we show that $\grb_i - \grg_{j_0} \in \Phi^+$ whenever $j_0 < i < i_0$.

Indeed, if $i < i_0$, then  $\beta_{i-1}+\grg_i = \grb_i \in \Phi^+$, therefore $\langle \beta_{i-1},\grg_i^\vee \rangle =-1$, and being $(\grg_i, \grg_j)=0$ for all $j < i$ it follows that $\langle \grb,\grg_i^\vee \rangle = -1$. Therefore, if $j_0 < i < i_0$, then it follows $\langle \grb_i, \grg_{j_0}^\vee \rangle = \langle \grb, \grg_{j_0}^\vee \rangle + \langle \grg_{j_0}, \grg_{j_0}^\vee \rangle = 1$, and the claim follows.
\end{proof}

\begin{lemma}	\label{lemma:antichain}
Suppose that $\Phi$ is simply laced. Let $\calB \in \Ort(\Phi_1^+)$ and suppose that it is not an antichain, then there exists $\calB' \in \Ort(\Phi_1^+)$ with $|\calB'| = |\calB|$ such that $\calB' \vdash \calB$, and $\dim \mathfrak a_{\calB'} < \dim \mathfrak a_\calB$.
\end{lemma}

\begin{proof}
Notice that $W_0$ acts on $\Ort(\Phi_1^+)$, we will find $\calB'$ in the $W_0$-orbit of $\calB$. Let $\grb \in \calB$ be a minimal element and let $\grb' \in \calB$ with $\grb < \grb'$. Write $\grb' - \grb = \grg_1 + \ldots + \grg_m$ for some pairwise orthogonal roots $\grg_1, \ldots, \grg_m \in \Phi^+$ as in Proposition \ref{prop:dominance-roots}, and notice that $\grg_i \in \Phi_0^+$ for all $i$. Set $\grg = \grg_1$. Then $\langle \grb' - \grb, \grg^\vee \rangle = 2$, and since $\Phi$ is simply laced it follows that $s_\grg(\grb) = \grb + \grg$ and $s_\grg(\grb') = \grb' - \grg$. On the other hand by Lemma~\ref{lem:adding-roots} $\grg$ is orthogonal to every root in $\calB \setminus \{\grb, \grb'\}$, therefore
$$
	s_\grg(\calB) = (\calB \setminus \{\grb, \grb'\}) \cup \{s_\grg(\grb), s_\grg(\grb')\}.
$$
Being $\grg < \grb' - \grb$, we have $\grb < s_\grg(\grb)$ and $\grb < s_\grg(\grb')$. Hence $s_\grg(\calB) \vdash \calB$, and since $\grb$ is minimal in $\calB$ we get $\grb \not \in \Psi(\mathfrak a_{\calB'})$.
\end{proof}

\begin{proposition}\label{prop:anticatene-simply}
Suppose that $\Phi$ is simply laced. Let $\calB \in \Ort(\Phi_1^+)$, then there is an antichain $\calA \in \Ort(\Phi_1^+)$ with $|\calA| = |\calB|$ such that $\calA \vdash \calB$.
\end{proposition}

\begin{proof}
Notice that $W_0$ acts on the orthogonal subsets of cardinality $m = |\calB|$. Suppose that $\calB$ is not an antichain, then by Lemma \ref{lemma:antichain} there is $\calB_1 \in \Ort_{(m,0)}(\Phi_1^+)$ such that $\calB_1 \vdash \calB$ and $\dim \mathfrak a_{\calB_1} < \dim \mathfrak a_\calB$. Let $i \geq 1$ and suppose that $\calB_1, \ldots, \calB_i \in \Ort_{(m,0)}(\Phi_1^+)$ are such that $\calB_i \vdash \ldots \vdash \calB_1 \vdash \calB$ and $\dim \mathfrak a_{\calB_i} < \dim \mathfrak a_{\calB_{i-1}} < \ldots < \dim \mathfrak a_\calB$. If $\calB_i$ is not an antichain, then we can apply Lemma \ref{lemma:antichain} again, and we find $\calB_{i+1} \in \Ort_{(m,0)}(\Phi_1^+)$ such that $\calB_{i+1} \vdash \calB_i$ and $\dim \mathfrak a_{\calB_{i+1}} < \dim \mathfrak a_{\calB_i}$. Since $\calB_{i+1}$ is not empty it must be $\dim \mathfrak a_{\calB_{i+1}} > 0$, therefore the process must stop for some $k$, and $\calB_k$ is an antichain.
\end{proof}

\subsection{The odd orthogonal case.} Consider the Hermitian pair $(B_n,\gra_1)$. We enumerate the set of simple roots $\Pi = \{\gra_1, \ldots, \gra_n\}$ as in \cite{Bou}. Given $i,j$ such that $1 \leq i \leq n$ and $1 \leq j < n$ we set
\begin{eqnarray*}
&\grb_i = \gra_1 + \ldots + \gra_i,\\
&\grb'_j = \gra_1 + \ldots + \gra_j + 2 \gra_{j+1} + \ldots + 2\gra_n.
\end{eqnarray*}
Then $\Phi^+_1 = \{\grb_i \st 1 \leq i \leq n\} \cup \{\grb'_j \st 1 \leq j < n\}$.
Notice that $\Phi^+_1$ contains a unique short root, namely $\grb_n$.

In this case $\Ort_{\max}(\Phi_1^+) = \{\calB_1, \ldots, \calB_n\}$, where we set  $\calB_i = \{\grb_i, \grb'_i\}$ for all $i < n$, and $\calB_n = \{\grb_n\}$. In particular, the only  possible types for an orthogonal subset are $(0,2)$, $(1,0)$ and $(0,1)$. Moreover $\calB_n \vdash \calB_{n-1} \vdash \ldots \vdash \calB_1$, and $\calB_n$ is the unique antichain in $\Ort_{\max}(\Phi_1^+)$. In particular, the following proposition trivially holds.

\begin{proposition}\label{prop:anticatene-tipoB}
Consider the Hermitian pair $(B_n,\gra_1)$, and let $\Phi^+ = \Phi_0^+ \cup \Phi_1^+$ be the corresponding decomposition.
\begin{itemize}
	\item[i)] Let $\calB \in \Ort(\Phi_1^+)$ of type $(h,k)$, then there is an antichain $\calA \subset \Phi_1^+$ of type $(h+\lfloor \frac{k}{2} \rfloor ,k-2\lfloor \frac{k}{2} \rfloor)$ such that $\calA \vdash \calB$.
\item[ii)] There exists a unique antichain $\calA_* \in \Ort_{\max}(\Phi_1^+)$, and $\calA_* \vdash \calB$ for all $\calB \in \Ort_{\max}(\Phi_1^+)$.
\end{itemize}
\end{proposition}

\subsection{The symplectic case.} Consider the Hermitian pair $(C_n,\gra_n)$. We enumerate the set of simple roots $\Pi = \{\gra_1, \ldots, \gra_n\}$ as in \cite{Bou}, and we embed $\Phi$ into the euclidean vector space $\mathbb R^n$ with orthonormal basis $\gre_1, \ldots, \gre_n$ by setting $\gra_i = \gre_i - \gre_{i+1}$ for all $i < n$ and $\gra_n = 2 \gre_n$. Then
$$
	\Phi = \{\pm(\gre_i \pm \gre_j) \st 1 \leq i, j \leq n\} \setminus \{0\},
$$
and $\Phi^+_1 = \{\gre_i + \gre_j \st 1 \leq i \leq j \leq n\}$. Notice that, for $1 \leq i \leq j \leq n$, we have
$$
	\gre_i + \gre_j = \gra_i + \ldots + \gra_{j-1} + 2\gra_j + \ldots + 2\gra_{n-1} + \gra_n.
$$
In particular $\gre_i + \gre_j \leq \gre_h + \gre_k$ if and only if $h \leq i$ and $k \leq j$. Notice that $\calS_{C_n,\gra_1} = \{2\gre_1, \ldots, 2\gre_n\}$, so that $r=n$.

Let $\calB \in \Ort(\Phi_1^+)$, and write $\calB = \{\gre_{i_1} + \gre_{j_1}, \ldots, \gre_{i_m}+\gre_{j_m}\}$ for some indices $i_1 \leq j_1, \ldots, i_m \leq j_m$. Correspondingly, we have a disjoint union $\bigcup_{k=1}^m \{i_k, j_k\}$, and $\calB \in \Ort_{\max}(\Phi_1^+)$ if and only if $\{1,\ldots,n\} = \bigcup_{k=1}^m \{i_k, j_k\}$, and it immediately follows that, if $\calB$ has type $(h,k)$, then $\calB \in \Ort_{\max}(\Phi_1^+)$ if and only if $2h+k=n$. Notice moreover that $\calB$ is an antichain if and only if, up to some permutation of $\{1, \ldots, m\}$, we have
$$ i_1 < i_2 < \ldots < i_{m-1} < i_m \leq j_m < j_{m-1} < \ldots < j_2 < j_1.$$
It follows that there is a unique $\calB \in \Ort_{\max}(\Phi_1^+)$ which satisfies the previous inequalities, therefore there is a unique antichain $\calA_* \in \Ort_{\max}(\Phi_1^+)$.
%\footnote{Ho aggiunto questo paragrafo per caratterizzare ortogonali massimali e osservare l'unicit\`a dell'anticatena.}

The following lemma is an easy consequence of previous description of $\Phi_1^+$.

\begin{lemma} \label{lemma:dominance-radiciC}
Let $\grb, \grb' \in \Phi_1^+$ be orthogonal roots.
\begin{itemize}
	\item[i)] Suppose that $\{\grb, \grb'\}$ is of type $(1,1)$ and suppose that $\grb < \grb'$. Then $\grb' - \grb = 2\gra + \gra'$ for some short roots $\gra, \gra' \in \Phi_0^+$ with $\langle \gra', \gra^\vee \rangle = -1$.
	\item[ii)] Suppose that $\{\grb, \grb'\}$ is of type $(2,0)$, and suppose that $\grb < \grb'$. Then $\grb' - \grb = \gra + \gra'$ for some orthogonal short roots $\gra, \gra' \in \Phi_0^+$.
	\item[iii)] Suppose that $\{\grb, \grb'\}$ is of type $(0,2)$, then $\grb - \grb' = 2\gra$ for some short root $\gra \in \Phi_0$.
\end{itemize}
\end{lemma}

\begin{proof}
Assume $\grb = \gre_i + \gre_j$ and $\grb' = \gre_h + \gre_k$, for some $i \leq j$ and $h \leq k$. The orthogonality implies that $i \neq h$ and $j \neq k$.

i) We have in this case $h \leq k < i \leq j$, and since $\{\grb, \grb'\}$ contains exactly one long root,  either $h = k < i < j$ or $h < k < i=j$. Therefore the claim follows by setting $\gra = \gre_k - \gre_i$ and $\gra' = \gre_h - \gre_k + \gre_i - \gre_j$.

ii) We have in this case $i < j$ and $h < k$, and the claim follows by setting $\gra = \gre_i - \gre_j = \gra_k + \ldots + \gra_{j-1}$ and $\gra' = \gre_h - \gre_k = \gra_h + \ldots + \gra_{k-1}$.

iii) We have in this case $i = j$ and $h = k$, and the claim follows by setting $\gra = \gre_h - \gre_i$.
\end{proof}

\begin{lemma}	\label{lemma:minimal-card}
Let $\calB \subset \Phi_1^+$ be an orthogonal subset of type $(h,k)$, set $k' = \lfloor \frac{k}{2} \rfloor$ and suppose that $\calB$ is not an antichain. Then there exists an orthogonal subset $\calB'$ of type $(h+k', k-2k')$ such that $\calB' \vdash \calB$, and $\dim \aa_{\calB'} < \dim \aa_\calB$.
\end{lemma}

\begin{proof}
Let $\grb \in \calB$ be a minimal element, and suppose that $\grb < \grb'$ for some $\grb' \in \calB$. We construct an orthogonal subset $\calB'$ such that $\calB' \vdash \calB$ and $\dim \aa_{\calB'} < \dim \aa_\calB$, whose type is $(h+1,k-2)$ if $\grb, \grb'$ are both long, and $(h,k)$ otherwise. Since two positive long roots in a root system of tyoe $C_n$ are always comparable, the claim will follow repeating the argument until $\calB'$ contains at most a single long root.

If $\grb, \grb'$ are both long, then by Lemma \ref{lemma:dominance-radiciC}  $\grb' - \grb = 2\gra$ for some short root $\gra \in \Phi_0^+$. Denote $\calB' = (\calB \setminus \{\grb, \grb'\}) \cup \{\gra+\grb\}$. Since $\gra+\grb = \grb' - \gra \in \Phi^+_1$, Lemma~\ref{lem:adding-roots} implies that $(\gra,\grb) = 0 $ for all $\grb \in \calB \setminus \{\grb, \grb'\}$. Therefore $\calB'$ is orthogonal, and it is of type $(h+1, k-2)$ since $\gra+\grb$ is a short root. Moreover $\calB' \vdash \calB$, and since $\grb$ is minimal in $\calB$ we get $\dim \aa_{\calB'} < \dim \aa_\calB$ as well.

Suppose that $\grb, \grb'$ are both short roots. Following Lemma \ref{lemma:dominance-radiciC}, write $\grb' - \grb = \gra + \gra'$ with $\gra, \gra' \in \Phi_0^+$ short orthogonal roots. In particular, it must be $\langle \grb', \gra^\vee \rangle = - \langle \grb, \gra^\vee \rangle = 1$, and by Lemma~\ref{lem:adding-roots} it follows $(\gra,\grb'') = 0$ for all $\grb'' \in \calB \setminus \{\grb, \grb'\}$. Therefore
$$s_\gra(\calB) = \left(\calB \setminus \{\grb, \grb'\} \right) \cup \{\gra+\grb, \grb'-\gra\}.$$
On the other hand, being $\grb' - \grb = \gra + \gra'$, we get $\grb < s_\gra(\grb)$ and $\grb < s_\gra(\grb')$. Therefore $s_\gra(\calB) \vdash \calB$, and since $\grb$ is minimal in $\calB$ it follows $\dim \aa_{s_\gra(\calB)} < \dim \aa_\calB$.

Suppose finally that $||\grb|| \neq ||\grb'||$. Following Lemma \ref{lemma:dominance-radiciC}, we can write $\grb' - \grb = 2\gra + \gra'$ where $\gra, \gra' \in \Phi_0^+$ are short roots with $\langle \gra', \gra^\vee\rangle = -1$. In particular we get $\gra + \gra' \in \Phi_0^+$, hence $\grb + \gra, \grb' - \gra \in \Phi_1^+$, and by Lemma~\ref{lem:adding-roots} it follows $(\gra, \grb'') = 0$ for all $\grb'' \in \calB \setminus \{\grb, \grb'\}$. Therefore
$$s_\gra(\calB) = (\calB \setminus\{\grb, \grb'\}) \cup \{s_\gra(\grb), s_\gra(\grb')\}.$$
On the other hand, being $\grb' - \grb = 2\gra + \gra'$, we get $\grb < s_\gra(\grb)$ and $\grb < s_\gra(\grb')$. Therefore $s_\gra(\calB) \vdash \calB$, and since $\grb$ is minimal in $\calB$ we get $\dim \aa_{s_\gra(\calB)} < \dim \aa_\calB$ as well.
\end{proof}

\begin{proposition}	\label{prop:anticatene-tipoC}
Consider the Hermitian pair $(C_n,\gra_1)$.
% and let $\Phi^+ = \Phi_0^+ \cup \Phi_1^+$ be the corresponding decomposition.
\begin{itemize}
	\item[i)] Let $\calB \subset \Phi_1^+$ be an orthogonal subset of type $(h,k)$. Then there is an antichain $\calA \subset \Phi_1^+$ of type $(h+ \lfloor \frac{k}{2} \rfloor,k-2 \lfloor \frac{k}{2} \rfloor)$ such that $\calA \vdash \calB$.
	\item[ii)] There exists a unique antichain $\calA_* \in \Ort_{\max}(\Phi_1^+)$, and $\calA_* \vdash \calB$ for all $\calB \in \Ort_{\max}(\Phi_1^+)$.
\end{itemize}
\end{proposition}

\begin{proof}
i) Suppose that $\calB$ is not an antichain, by Lemma \ref{lemma:minimal-card} there is an orthogonal subset $\calB_1 \subset \Phi_1^+$ of type $(h+\lfloor \frac{k}{2} \rfloor, k-2\lfloor \frac{k}{2} \rfloor)$ such that $\calB_1 \vdash \calB$ and $\dim \aa_{\calB_1} < \dim \aa_\calB$. Suppose that $\calB_i$ is defined, and suppose that $\calB_i$ is not an antichain. Then we can apply Lemma \ref{lemma:minimal-card} again, and we find an orthogonal subset $\calB_{i+1} \subset \Phi_1^+$ such that $\calB_{i+1} \vdash \calB_i$ and $\dim \aa_{\calB_{i+1}} < \dim \aa_{\calB_i}$. Since $\calB_{i+1}$ is not empty, $\aa_{\calB_{i+1}}$ cannot be zero, therefore the process must stop for some $k$, and $\calB_k$ is an antichain.

ii) As we already noticed, if $\calB \in \Ort(\Phi_1^+)$ has type $(h,k)$, then $\calB$ is maximal if and only if $2h+k=n$. Therefore by i) for all $\calB\in \Ort_{\max}(\Phi_1^+)$ there exists an antichain $\calA\in \Ort_{\max}(\Phi_1^+)$ such that $\calA \vdash \calB$. We also already noticed that there is a unique antichain $\calA_*\in \Ort_{\max}(\Phi_1^+)$, therefore $\calA_* \vdash \calB$ for all $\calB \in \Ort_{\max}(\Phi_1^+)$.
\end{proof}

\subsection{Hermitian symmetric spaces of tube type} \label{tt}
Let $(\Pi, \gra_q)$ be a Hermitian pair and let $\gop^+$ be the corresponding standard parabolic subalgebra of $\gog$. Let $\calS_{\Pi,\gra_q} = \{\grg_1, \ldots, \grg_r\}$ be the set of Harish-Chandra strongly orthogonal roots, and set $\h^- = \Span(\gamma_i^\vee\st i=1,\ldots,r)$. By \cite{HC}, \cite{M}, a root $\a\in\D$ is in $\Dp_1$ if and only if either $\gra_{|\h^-} = \half(\gamma_i+\gamma_j)$ , for some $i \leq j$, or $\gra_{|\h^-} = \half \gamma_i$ for some $i$. A root $\a\in\D$ is in $\Dp_0$ if and only if either $\gra_{|\h^-} = \half(\gamma_i-\gamma_j)$, for some $i \leq j$, or $\gra_{|\h^-} =\pm \half \gamma_i$, for some $i$. Recall that the Hermitian symmetric space $G/P$ is called of \textit{tube type} if it is holomorphically equivalent to the tube over a self dual cone. It is known (cf. \cite{KW}) that Hermitian symmetric spaces of tube type correspond to Hermitian involutions such that $\a\in\D$ is in $\Dp_1$ if and only if $\a_{|\h^-} = \half(\gamma_i+\gamma_j)$ for some $i \leq j$, and a root $\a\in\D$ is in $\Dp_0$ if and only if $\gra_{|\h^-} = \half(\gamma_i-\gamma_j)$, for some $i \leq j$. We will call such involutions \textit{Hermitian involutions of tube type}. Observe that a Hermitian involution is of tube type if and only if 
\begin{equation}\label{sumtube}
(\sum_{i=1}^r \gamma_i,\a)=(\a_q,\a_q)\ i,\mbox{  for all }\a\in\D_i, \ i=0,1. 
\end{equation}
%\footnote{Referenza?}

Hermitian symmetric spaces of tube type are classified by the Hermitian pairs $(\Pi, \gra_q)$ such that $w_0(\gra_q) = -\gra_q$, in which case we say that $(\Pi, \gra_q)$ is a \textit{Hermitian pair of tube type} (see e.g. \cite[Ch. X, D.4 pg. 528]{Hel}). In particular, we have the following possibilities:
\begin{itemize}
	\item[i)] $(A_{2q-1}, \a_q)$;
	\item[ii)] $(B_n, \a_1)$;
	\item[iii)] $(C_n, \a_n)$;
	\item[iv)]  $(D_n,\a_{n-1})$ with $n$ even; $(D_n,\a_n)$ with $n$ even; $(D_n, \a_1)$ for all $n$;
 	\item[v)] $(E_7,\a_7)$.
\end{itemize}
Notice that being of tube type  is equivalent to the fact that $\gop_\mru^+$ is a regular pre-homogeneous space under the action of $G_0$, namely the boundary of the open $G_0$-orbit has codimension 1 (see \cite{MRS} and the references therein).

%These are the cases corresponding to the following Cartan labels, and to the following Hermitian pairs $(\Pi,\gra_p)$, that we call \textit{Hermitian pairs of tube type}. it follows that an Hermitian pair 
%\begin{itemize}
%	\item[i)] $AI\!I\!I$ ($p=q$), corresponding to the Hermitian pair $(A_{2p-1}, \a_p)$;
%	\item[ii)] $DI\!I\!I$ ($n$ even), corresponding to the Hermitian pairs $(D_n, \a_{n-1})$ and $(D_n,\a_n)$, with $n$ even;
%	\item[iii)] $CI$, corresponding to the Hermitian pair $(C_n,\a_n)$;
%	\item[iv)] $B\!DI$ ($p=2$), corresponding to the Hermitian pairs $(B_n,\a_1)$ and $(D_n,\a_1)$;
%	\item[v)] $EV\!I\!I$, corresponding to the Hermitian pair $(E_7,\a_7)$.
%\end{itemize}

%We now prove that for Hermitian involutions of tube type there exists a unique antichain in $\Ort_{\max}(\Phi_1^+)$. We start with a simple observation:
If $(\Pi,\a_q)$ is a Hermitian pair of tube type and $\Pi$ is not simply laced, then the short roots in $\Phi_1$ admit a nice description:
\begin{lemma}\label{short:tube}
Suppose that $\Phi$ is not simply laced and let $\s$ be a Hermitian involution of tube type. Let $\calS$ be an orthogonal subset of $\Phi_1^+$ of maximal cardinality and let $\be\in \Phi_1^+$ be a short root. Then $\be=\half(\gamma +\gamma')$ for some distinct elements $\grg, \grg' \in \calS$.
\end{lemma} 
\begin{proof}
By Remark \ref{oss:max-cardinality}, every $\grg \in \calS$ is a long root of $\Phi$, and we can choose a set of positive roots in $\Phi_0^+ \subset \Phi_0$ so that $\calS$ is the corresponding set of Harish-Chandra strongly orthogonal roots. Since $\s$ is of tube type, we have that $\be=\half(\gamma+\gamma')+\l$, for some $\grg, \grg' \in \calS$ with $(\grg, \grg') = 0$ and some $\l$ with $(\l,\gamma'')=0$ for all $\grg'' \in \calS$, therefore $\Vert\be\Vert^2=\frac{1}{4}( \Vert\gamma \Vert^2+\Vert\gamma' \Vert^2)+\Vert\l\Vert^2=\half \Vert\gamma \Vert^2+\Vert\l\Vert^2$. On the other hand $\be$ is a short root, therefore $||\grb||^2 = \half ||\grg||^2$ and it follows $\grl = 0$.
\end{proof}

If $\grs$ is the Hermitian involution of tube type associated to the Hermitian pair $(B_n,\gra_1)$ or $(C_n,\gra_n)$, we proved in Proposition \ref{prop:anticatene-tipoB} and Proposition \ref{prop:anticatene-tipoC} that there exists a unique antichain $\calA_* \in \Ort_{\max}(\Phi_1^+)$, and that $\calA_* \vdash \calB$ for all $\calB \in \Ort_{\max}(\Phi_1^+)$. We now show that this property holds whenever $\grs$ is a Hermitian involution of tube type.

\begin{proposition} \label{lemmahermitiano}
Suppose that $\s$ is a Hermitian involution of tube type.
% and let $\Phi^+ = \Phi_0^+ \cup \Phi_1^+$ be the corresponding decomposition. 

Then there exists a unique antichain $\mathcal A_* \in \Ort_{\max}(\Phi_1^+)$, and $\calA_* \vdash \calB$ for all $\calB \in \Ort_{\max}(\Phi_1^+)$.
\end{proposition}

\begin{proof}
As we noticed, the claim has already been proved if $\Phi$ is not simply laced. Therefore we will assume that $\Phi$ is simply laced, so that $\Ort_{\max}(\Phi_1^+)$ coincides with the collection of the orthogonal subsets of maximal cardinality $r$.  By Proposition \ref{prop:anticatene-simply}, for all $\calB \in \Ort_{\max}(\Phi_1^+)$, there is an antichain $\calA \in \Ort_{\max}(\Phi_1^+)$ such that $\calA \vdash \calB$. Therefore we only need to show the uniqueness of the antichain in $\Ort_{\max}(\Phi_1^+)$. 
%\footnote{Il vecchio corollario non mi sembra che servisse, questo paragrafo risponde.}

%\footnote{Da qui in poi assumerei simply laced. Siccome per l'esistenza dell'anticatena abbiamo distinto simply laced, B e C secondo me non ha molto senso complicare molto la dimostrazione per fare l'unicit\`a senza riferirsi a quei casi.}

Suppose $\mathcal A, \mathcal A' \in \Ort_{\max}(\Phi_1^+)$ are both antichains. %We will prove that $\mathcal A=\mathcal A'$ by induction on $r$. If $r=1$, since $\s$ corresponds to a Hermitian symmetric space of tube type, it follows that $\Dp_1=\{\gamma_1\}$ consists of a single root, so the unique antichain is $\Dp_1$ itself.
Since they are both of maximal cardinality, by Remark \ref{oss:max-cardinality}, they are both sets of Harish-Chandra roots for some choice of positive sets of roots in $\Phi_0$. Since the pair $(\Pi,\a_q)$ is of tube type we have, by (\ref{sumtube}),
$$
(\sum_{\gamma\in\calA}\grg,\a_q)=(\a_q,\a_q)=(\sum_{\gamma'\in\calA'}\grg',\a_q)
$$
and, if $\a\in\Pi_0$,
$$
(\sum_{\gamma\in\calA}\grg,\a)=0=(\sum_{\gamma'\in\calA'}\grg',\a).
$$
It follows that
\begin{equation}\label{sumgrg}
\sum_{\gamma\in\calA}\grg=\sum_{\gamma'\in\calA'}\grg'.
\end{equation}
%Suppose that $r > 1$. 
%We choose $\Gamma=\{\gamma_1,\ldots, \gamma_r\}\in \Ort_{\max}(\Phi_1^+)$ and $\Gamma'=\{\gamma'_1,\ldots, \gamma'_r\}\in \Ort_{\max}(\Phi_1^+)$ such that
%\begin{gather*}
%\mathcal A=\{\half(\gamma_{2i-1}+\gamma_{2i}) \st i=1,\ldots,h\}\cup\{\gamma_{2h+i} \st i=1,\ldots,k\}\\
%\mathcal A'=\{\half(\gamma'_{2i-1}+\gamma'_{2i}) \st i=1,\ldots,h'\} \cup \{\gamma'_{2h'+i} \st i=1,\ldots,k'\}.
%\end{gather*} 
Let $C$ be the matrix $((\gamma,\gamma'))_{\gamma\in\mathcal A,\gamma' \in \mathcal A'}$. 
 Consider the  matrix $C'$ obtained by replacing the nonzero entries of $C$ with $1$. This is the incidence matrix of a relation. Let $\mathcal G$ be its incidence graph. Write $\mathcal G=\cup_i \mathcal G_i$, where $\mathcal G_i$ are the connected components of $\mathcal G$. Assume first that $\mathcal G_i$  has more than one node. If $\grg\in\mathcal G_i\cap\calA\cap\calA'$ then $(\grg,\grg')=0$ for all $\grg'\in\calA'\setminus\{\gamma\}$, so $\grg$ is connected in $\mathcal G$ only to itself. Since $\mathcal G_i$ is connected  this is not possible, hence  $\mathcal G_i\cap\calA\cap\calA'=\emptyset$. If $\mathcal G_i\cap\calA=\emptyset$ then any $\grg'\in\mathcal G_i\cap\calA'$ is orthogonal to $\calA$, contradicting the fact that $\calA$ is in $\Ort_{\max}(\Phi_1^+)$. Symmetrically we have that also   $\mathcal G_i\cap\calA'$ is not empty. Since $\mathcal G_i$ is connected, if $\gamma_0\in\mathcal G_i\cap\calA$, there must be $\gamma'_0\in\mathcal G_i\cap\calA'$ such that $\gamma_0\ne\grg'_0$ and $(\grg_0,\grg'_0)\ne0$.  
 Since 
 $$
 \langle\sum_{\gamma\in\calA}\gamma,\grg_0^\vee\rangle=2= \langle\sum_{\gamma'\in\calA'}\gamma',\grg_0^\vee\rangle=1+\langle\sum_{\gamma'\ne\grg'_0}\gamma',\grg_0^\vee\rangle,
 $$
 we see that there are exactly two nodes  in $\mathcal G_i$ to which $\grg_0$ is connected. Symmetrically, the same property holds for all $\gamma'\in\mathcal G_i\cap\calA'$. Thus every node has degree exactly  $2$ in $\mathcal G_i$. It follows that $\mathcal G_i$ is a cycle
$$
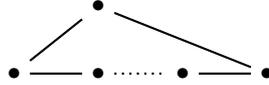
\begin{tikzpicture}[double distance=2pt,>=stealth] 
\matrix [matrix of math nodes,row sep=.5cm,column sep=.7cm]
{&|(11)|\bullet&&&[-1cm]&[-0.5cm]&\\
|(22)|\bullet&|(23)|\bullet&|(24)|\bullet& |(25)|\bullet\\
};
\begin{scope}[every node/.style={midway,auto},thick]
\draw (11) -- (22);
\draw (11) -- (25);
\draw (22) -- (23);
\draw[dotted] (23) -- (24);
\draw (24)--(25);
\end{scope}
\end{tikzpicture}
$$
Since $\mathcal G_i\cap\calA\cap\calA'=\emptyset$ and the nodes in $\mathcal A$ connect only to nodes in $\mathcal A'$ we have only this possibility (letting $\circ$ be the nodes in $\mathcal A$ and $\bullet$ the nodes in $\mathcal A'$):
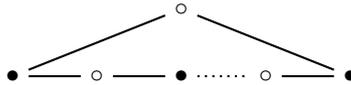
$$
\begin{tikzpicture}[double distance=2pt,>=stealth] 
\matrix [matrix of math nodes,row sep=.5cm,column sep=.7cm]
{&&|(11)|\circ&\\
|(22)|\bullet&|(23)|\circ&|(24)|\bullet& |(25)|\circ& |(26)|\bullet\\
};
\begin{scope}[every node/.style={midway,auto},thick]
\draw (11) -- (22);
\draw (11) -- (26);
\draw (22) -- (23);
\draw(23) -- (24);
\draw[dotted](24) -- (25);
\draw (25)--(26);
\end{scope}
\end{tikzpicture}
$$

 If $(\grg,\grg')\ne0$ then $(\grg,\grg')>0$ so $\grg-\grg'$ is a root. It follows that either $\grg>\grg'$ or $\grg'>\grg$. We give an orientation to the graph $\mathcal G$ by orienting the edges so that they point from the larger to the smaller root. Since both $\mathcal A$ and $\mathcal A'$ are antichains, we cannot have consecutive arrows $
\begin{tikzpicture}[double distance=2pt,>=stealth] 
\matrix [matrix of math nodes,row sep=1cm,column sep=.4cm]
{%|(11)|\circ&&&[-1cm]&[-0.5cm]&\\
|(22)|\circ&|(23)|\bullet& |(25)|\circ\\
};
\begin{scope}[every node/.style={midway,auto},thick]
\draw[->] (22) -- (23);
\draw[->] (23)--(25);
\end{scope}
\end{tikzpicture}$, $
\begin{tikzpicture}[double distance=2pt,>=stealth] 
\matrix [matrix of math nodes,row sep=.1cm,column sep=.4cm]
{%|(11)|\circ&&&[-1cm]&[-0.5cm]&\\
|(22)|\bullet&|(23)|\circ& |(25)|\bullet\\
};
\begin{scope}[every node/.style={midway,auto},thick]
\draw[->] (22) -- (23);
\draw[->] (23)--(25);
\end{scope}
\end{tikzpicture}$. It follows that the nodes $\circ$ are either all sources or all sinks. By eventually exchanging $\calA$ and $\calA'$, we can assume  that all $\circ$ are sources and all $\bullet$ are sinks. This means that we can enumerate $\calA\cap\mathcal G_i=\{\grg_1,\ldots,\grg_s\}$ and $\calA'\cap\mathcal G_i=\{\grg'_1,\ldots,\grg'_s\}$ so that $\grg'_i<_0\grg_i$.
 
 It follows that $\l=\sum \grg_i-\sum\grg'_i >_0 0$. On the other hand
\begin{eqnarray*}
\Vert\l\Vert^2&=&\Vert\sum \grg_i\Vert^2+\Vert\sum \grg_i'\Vert^2-(\sum \grg_i,\sum\grg'_i)-(\sum \grg'_i,\sum\grg_i)\\
&=&\Vert\sum \grg_i\Vert^2+\Vert\sum \grg_i'\Vert^2-(\sum \grg_i,\sum_{\grg'\in\calA'}\grg')-(\sum \grg'_i,\sum_{\grg\in\calA}\grg).
\end{eqnarray*}
By (\ref{sumgrg}), 
\begin{eqnarray*}
\Vert\l\Vert^2&=&\Vert\sum \grg_i\Vert^2+\Vert\sum \grg_i'\Vert^2-(\sum \grg_i,\sum_{\grg\in\calA}\grg)-(\sum \grg'_i,\sum_{\grg'\in\calA'}\grg')\\
&=&\Vert\sum \grg_i\Vert^2+\Vert\sum \grg_i'\Vert^2-(\sum \grg_i,\sum \grg_i)-(\sum \grg'_i,\sum \grg'_i)=0.
\end{eqnarray*}
 It follows that $\mathcal G_i$ has only one node for all $i$. As observed earlier $\mathcal G_i\cap \calA\ne \emptyset$ and $\mathcal G_i\cap\calA'\ne \emptyset$, thus, if $\mathcal G_i=\{\gamma\}$,  then $\grg\in\calA\cap\calA'$. Therefore $\calA\cup\calA'=\cup_i\mathcal G_i\subset\calA\cap\calA'$. Thus $\calA=\calA'$.
 \end{proof}
\begin{remark} The proof of uniqueness of the antichain we have given when $\Phi$ is simply laced can be extended (with some complications) to a uniform proof for any $\Phi$. Since the two non simply laced cases are easily dealt individually, we preferred to omit this more complicated approach.
\end{remark}
Given $\calB \subset \Phi_1^+$, 
set
$$
	\calB^{\leq_0} = \{ \gra \in \Phi^+_1 \st \mbox{ there is }\grb \in \calB \mbox{ such that }\gra \leq_0 \grb\}.
$$
Notice that $\calB^{\leq_0} = \Psi(\aa^-_\calB)$, where $\aa_{\calB}^- \subset \gop_\mru^+$ is the $B_0^-$-stable subalgebra generated by $\calB$ and $B_0^- \subset G_0$ is the opposite Borel subgroup of $B$.

If $\calA_* \in \Ort_{\max}(\Phi_1^+)$ is the unique antichain, it follows by Proposition \ref{lemmahermitiano}, that $\calA_* \subset \calB^{{\geq}_0}$ for all $\calB \in \Ort_{\max}(\Phi_1^+)$. The following corollary shows that  $\calA_* \subset \calB^{\leq_0}$ as well.

\begin{corollary}	\label{cor:anticatena-sottosopra}
Let $\calA_* \in \Ort_{\max}(\Phi_1^+)$ be the unique antichain. Then $\calA_* \subset \calB^{\leq_0}$ for all $\calB \in \Ort_{\max}(\Phi_1^+)$. 
\end{corollary}

\begin{proof}
Let $\leq_0'$ be the partial order on $\Phi_1^+$ defined by $\Phi_0^-$. Then $\leq_0'$ is the reverse partial order of $\leq_0$, therefore a subset $\calA \subset \Phi_1^+$ is an antichain w.r.t. $\leq_0$ if and only if it is an antichain w.r.t. $\leq_0'$. Therefore, if $\vdash'$ is the preorder on $\Ort(\Phi_1^+)$ defined by $\leq_0'$, it follows by Proposition \ref{lemmahermitiano} that $\calA_* \vdash' \calB$, namely $\calA_* \subset \calB^{\leq_0}$.
\end{proof}

\section{The special $B_0$-stable abelian subalgebra}\label{5}
%%%%%%%%%%%%%%%%%%%%%%%%%%%%%%%%%%%%%%%%%%%%%%%%%%%%%%%
%%%%%%%%%%%%%%%%%%%%%%%%%%%%%%%%%%%%%%%%%%%%%%%%%%%%%%%

For the rest of the paper we will assume that $\s : G \lra G$ is an (indecomposable) involution. Moreover, throughout this section, we assume that $\g_0$ is semisimple and that the simple root $\a_p \in \widehat{\Pi}$ corresponding to $\s$ is long and non-complex.

Recall from Section \ref{bs} the element $w_p$ of $\Wab$ and the corresponding subalgebra $\aa_p\in\mathcal I_{ab}^\s$.
% be the abelian $B_0$-stable subalgebra corresponding to $w_p$ under the map defined in Proposition \ref{imrn}. 
We call $\aa_p$ the {\sl special $B_0$-stable abelian subalgebra}. 
%\footnote{Forse dovremmo controllare se $\aa_p$ \`e la stessa sottoalgebra che produce Panyushev come controesempio alla sfericit\`a.}

For each component $\Sigma$ of $\Pi_0$, we let $\D(\Sigma)$ be the root subsystem of $\Phi_0$ generated by $\Sigma$. As shown in \cite[Lemma 5.7]{CMP}, there is a unique simple root $\a_\Sigma \in \Sigma$ which is connected to $\a_p$. If moreover $\theta_\Sigma$ is the highest root in $\Phi(\Sigma)$, then $[\theta_\Sigma:\a_\Sigma]=1$, therefore $(\grS, \gra_\grS)$ is a Hermitian pair. It is then clear that $\gamma \in \D(\Sigma)$ is orthogonal to $\a_p$ if and only if  $[\gamma:\a_\Sigma]=0$.
Following Section \ref{4}, we denote by $\Phi(\Sigma)^+_1$ the set of roots in $\Phi(\Sigma)^+$ that have $\a_\Sigma$ in their support. %\footnote{Ho cambiato $\Phi^+_1(\Sigma)$ in $\Phi(\Sigma)^+_1$, mi sembra pi\`u chiaro se no sembra che sia un sottoinsieme di $\Phi_1^+$.} 

Set $\mathcal C_{\s} =\{\a\in\Dap \st \a_p+k\d-\a\in\Dap\}$ and define
$$\mathcal C^1_{\s}= \mathcal C_{\s} \cap \Da_1 = \{\a\in\mathcal C_\s \st  [\a:\a_p]=1\}.$$
Notice that $\mathcal C^1_{\s}\subset \Dap_{re}$.

\begin{lemma}\label{N(w*p)} \ 
\begin{itemize}
	\item[i)] $N(w_p)=\bigcup_{\Sigma}\{\gamma+\a_p\st \gamma\in\Phi(\Sigma)^+_1\} \cup \{\a_p\}$, where $\grS$ ranges  among the components of $\Pi_0$;
	\item[ii)] $N(w_p)=\mathcal C^1_{\s}$;
	\item[iii)] The map $\Upsilon:\mathcal C^1_{\s} \to \bigcup_\Sigma \Phi(\Sigma)^+_1 \cup \{0\}$ mapping $\eta$ to $\eta-\a_p$ is an order preserving bijection, where $\grS$ varies among the components of $\Pi_0$;
	\item[iv)] If $\gamma,\eta \in \mathcal C^1_{\s}\setminus \{\a_p\}$, then $(\Upsilon(\gamma),\Upsilon(\eta))=(\gamma,\eta)$.
	\item[v)] If $\eta \in \mathcal C^1_{\s}\setminus \{\a_p\}$, then $(\gra_p, \Upsilon(\eta)) = -(\gra_p,\eta)$.
	% \footnote{Ho aggiunto questo che mi sembrava venisse usato pi\'u volte nel seguito.}
\end{itemize}
\end{lemma}

\begin{proof} 
i) It is well known that $N(w_{0,\a_p}w_0)=\bigcup_{\Sigma}\Phi(\Sigma)^+_1$. If  $\gamma\in\Phi(\Sigma)^+_1$, then $\langle \gamma,\a_p^\vee \rangle = \langle \a_\Sigma,\a_p^\vee\rangle$, and since $\a_p$ is long. we have   $\langle \a_\Sigma,\a_p^\vee \rangle =-1$. It follows that $s_p(\gamma)=\gamma+\a_p\in\Dap$, hence
$$
N(w_p)=N(s_pw_{0,\a_p}w_0)=\{\a_p\}\cup s_p(N(w_{0,\a_p}w_0)),$$
and the claim follows.

ii) Clearly $\a_p\in \mathcal C^1_{\s}$. Moreover, if $\gamma \in\Phi(\Sigma)^+_1$, then $k\d+\a_p-(\gamma+\a_p)=k\d-\gamma\in\Dap$. It follows that $N(w_p)\subseteq \mathcal C^1_{\s}$.
Since $w_p(\a_p)=k\d+\a_p$, it follows that $\ell(w_ps_p)=\ell(w_p)+1$. This implies that $N(w_ps_p)=N(w_p)\cup\{k\d+\a_p\}$.  Let $\eta\ne\a_p$ be in $ \mathcal C^1_{\s}$ so that there is $\beta$  such that $\eta+\beta=k\d+\a_p$. By biconvexity of $N(w_ps_p)$, exactly one between $\eta$ and $\beta$ is in $N(w_p)$. Since $[\eta:\a_p]=1$ and $[\beta:\a_p]=2$, it follows from the fact that $w_p\in\Wab$ that $\eta\in N(w_p)$. Thus $\mathcal C^1_{\s}\subseteq N(w_p)$, and the claim follows.

iii) Follows by i) and ii).

iv) Since $\a_p$ is long, if $\a\in\Phi(\Sigma)^+_1$ then $(\a,\a_p)=-\frac{(\a_p,\a_p)}{2}$. Therefore, if $\grg, \eta \in \mathcal C^1_{\s}\setminus \{\a_p\}$, we have 
$$
(\gamma,\eta)=(\gamma-\a_p,\eta-\a_p)+(\gamma-\a_p,\a_p)+(\eta-\a_p,\a_p)+(\a_p,\a_p)=(\gamma-\a_p,\eta-\a_p).
$$
The second claim follows as well, since $(\Upsilon(\eta), \Upsilon(\eta)) = (\eta, \eta)$.

v) Since $\gra_p$ is long and $\Upsilon(\eta) \in \Phi(\grS)^+_1$, it must be $\langle \Upsilon(\eta),\gra_p^\vee \rangle = -1$. Therefore $\langle \eta,\gra_p^\vee \rangle = 1$, and it follows $(\gra_p, \Upsilon(\eta)) = -(\gra_p,\eta)$.
\end{proof}

\begin{lemma}\label{affine type}
 Let $\{\eta_1,\dots,\eta_t\}$ be an orthogonal set of real roots, let $\eta_{t+1}$ be a real root such that $(\eta_i,\eta_{t+1})<0$ for all $i\leq t$, and set $A=(\langle \eta_j, \eta_i^\vee\rangle)_{i,j= 1, \dots, t+1}$. Then $A$ is a generalized Cartan matrix of finite or affine type.
\end{lemma}
 
\begin{proof}
The fact that $A$ is a generalized Cartan matrix \cite[Section 1.1] {Kac} is clear. It is also symmetrizable: setting $D=\mathrm{diag}(||\eta_1||^2,\ldots,||\eta_{t+1}||^2)$ we have that $DA=2((\eta_i,\eta_j))$, which is symmetric. Since $(\eta_i, \eta_{t+1})\ne 0$ for all $i\leq t$, $DA$ is an indecomposable matrix. Therefore it is enough to check that $DA$ is of finite or affine type. By \cite[Lemma 4.5]{Kac}, we need to check that $DA$ is positive semi-definite of corank less than or equal to $1$. Since $\eta_1,\ldots, \eta_t$ are orthogonal $DA$ has rank at least $t$. It is clear that, given a positive semi-definite symmetric bilinear  form $(\cdot,\cdot)$ on a vector space $V$ and a set of vectors $R=\{v_1,\ldots,v_k\}$ in $V$, the matrix $((v_i,v_j))$ is positive semi-definite. On the other hand, since $\Da$ is an affine system, the invariant form $(\cdot,\cdot)$ is positive semi-definite, and being $DA=2((\eta_i,\eta_j))$ it follows that $DA$ is positive semi-definite.
\end{proof}

Let $\calS \subset \calC^1_\s \setminus \{\gra_p\}$ be an orthogonal subset of maximal cardinality. Consider the set of roots
$$
\Pi_\calS = \calS \cup \{-\a_p\}
$$
and the matrix $A_\calS = (\langle \eta', \eta^\vee \rangle)_{\eta, \eta' \in \Pi_\calS}$. We will show that $G_0\aa_p$ is not spherical. The following is the key result in this direction. %\footnote{Siccome ci sono varie sottoalgebre non sferiche il pedice $ns$ non mi convinceva molto, anzi direi che i vari sistemi di radici ``non sferici'' dipendano dall'insieme ortogonale considerato}

\begin{proposition}\label{63}
Let $\calS \subset \calC^1_\s \setminus \{\gra_p\}$ be an orthogonal subset of maximal cardinality. Then $|\calS| \le 4$ and $A_\calS$ is a Cartan matrix of affine type.
\end{proposition}

\begin{proof}
%\footnote{Ho eliminato qualche diagramma da questa dimostrazione, in vari casi mi sembrava pi\`u chiaro dare semplicemente la coppia Hermitiana.}
Set $\calS = \{\eta_1, \ldots, \eta_t\}$. By Lemma \ref{affine type}, $A_\calS$ has to be either of finite or of affine type. By a slight abuse of notation we denote by $\Pi_\calS$ the  corresponding Dynkin diagram. This diagram has $t+1$ nodes with $t$ nodes connected only to the remaining node corresponding to $\a_p$. This immediately implies that $t\le 4$, and since $\a_p$ is long, the node connected to all other nodes corresponds to a long simple root.

If $t=4$ then the diagram is of type $D_4^{(1)}$, hence it is affine. If $t=3$ the only possibilities are $D_4$ or $B_3^{(1)}$. If $\Pi_\calS$ is of type $D_4$, then $\eta=k\d-(\eta_1+\eta_2+\eta_3-2\a_p)$ is a root in $\Da$ with $[\eta:\a_p]=1$. Set $\beta=\eta_1+\eta_2+\eta_3-\a_p\in\Da$. Then $\be\in\Da$, $[\be:\a_p]=2$, and $\eta+\be=k\d+\a_p$. It follows that $\eta\in \mathcal C^1_{\s}$. Being $(\eta,\eta_i)=0$ for $i=1,2,3$, we see that $\{\eta_1,\eta_2,\eta_3\}$ is not a set of maximal cardinality in $\mathcal C^1_{\s}\setminus \{\a_p\}$. This excludes the possibility that $\Pi_\calS$ is of type $D_4$, so $\Pi_\calS$ is of affine type.

If $t=2$ then $\Pi_\calS$ can be only of type $A_3$, $B_3$, $G_2^{(1)}$, $D_3^{(2)}$. By Lemma \ref{N(w*p)} together with Remark \ref{oss:max-cardinality}, $\{\Upsilon(\eta_1),\Upsilon(\eta_2)\}$ is an orthogonal subset of maximal cardinality in $\bigcup_{\Sigma}\Phi(\Sigma)^+_1$, and both $\Upsilon(\eta_1)$ and $\Upsilon(\eta_2)$ are long roots in the respective components of $\Pi_0$. In particular, $\Pi_0$ contains at most two components.

Suppose that $t=2$ and that $\Pi_0 = \Sigma_1 \cup \Sigma_2$ is the union of two components, then by Remark \ref{oss:max-cardinality} the Hermitian symmetric spaces corresponding to $(\Sigma_1, \a_{\Sigma_1})$ and $(\Sigma_2, \a_{\Sigma_2})$ have both rank 1, hence $\Sigma_1$, $\Sigma_2$ are both of type $A$, and $\a_{\grS_i}$ is an extremal root in $\grS_i$ (see Table \ref{table:hermitian-pairs}).
%\footnote{Qui stiamo usando i ranghi delle variet\'a simmetriche Hermitiane, li ho richiamati in una tabella.}
%$$
%\begin{tikzpicture}[double distance=2pt,>=stealth] 
%\matrix [matrix of math nodes,row sep=1cm,column sep=.7cm]
%{%|(11)|\circ&&&[-1cm]&[-0.5cm]&\\
%|(22)|\circ&|(23)|\circ&|(24)|\circ& |(25)|\bullet\\
%};
%\begin{scope}[every node/.style={midway,auto},thick]
%\draw (22) -- (23);
%\draw[dotted] (23) -- (24);
%\draw (24)--(25);
%\end{scope}
%\end{tikzpicture}
%$$
%where the roots $\a_{\Sigma_i}$ are denoted by black nodes, here and in the following.
Since $\gra_{\grS_i}$ and $\Upsilon(\eta_i)$ are both long in $\Sigma_i$, by Lemma \ref{N(w*p)} v) it follows that
\begin{equation}\label{sameedges}
\langle -\a_p,\eta_i^\vee \rangle = \langle \a_p,\Upsilon(\eta_i)^\vee\rangle = \langle \a_p,\a_{\Sigma_i}^\vee\rangle.
\end{equation}
This means that $\a_p$ is connected to $\a_{\Sigma_i}$ in $\Pia$ with the same number of edges that connect $-\a_p$ to $\eta_i$ in $\Pi_\calS$. If $\Pi_\calS$ is of type $A_3$, it follows that $\Pia$ is of type $A$, of the shape
$$
\begin{tikzpicture}[double distance=2pt,>=stealth] 
\matrix [matrix of math nodes,row sep=1cm,column sep=.7cm]
{%|(11)|\circ&&&[-1cm]&[-0.5cm]&\\
|(1)|\circ&|(2)|\circ&|(3)|\circ& |(4)|\bullet&|(5)|\circ&|(6)|\bullet&|(7)|\circ&|(8)|\circ&|(9)|\circ\\
};
\begin{scope}[every node/.style={midway,auto},thick]
\draw (1) -- (2);
\draw[dotted] (2) -- (3);
\draw (3)--(4);
\draw (4)--(5);
\draw (5)--(6);
\draw (6) -- (7);
\draw[dotted](7) -- (8);
\draw (8) -- (9);
\end{scope}
\end{tikzpicture}
$$
(where the roots $\a_{\Sigma_i}$ are denoted by black nodes), which is absurd since $\Pia$ is not of finite type. If $\Pi_\calS$ is of type $B_3$, then $\Pia$ is of the shape
$$
\begin{tikzpicture}[double distance=2pt,>=stealth] 
\matrix [matrix of math nodes,row sep=1cm,column sep=.7cm]
{%|(11)|\circ&&&[-1cm]&[-0.5cm]&\\
|(1)|\circ&|(2)|\circ&|(3)|\circ& |(4)|\bullet&|(5)|\circ&|(6)|\bullet&|(7)|\circ&|(8)|\circ&|(9)|\circ\\
};
\begin{scope}[every node/.style={midway,auto},thick]
\draw (1) -- (2);
\draw[dotted] (2) -- (3);
\draw (3)--(4);
\draw (4)--(5);
\draw[double,->] (5)--(6);
\draw (6) -- (7);
\draw[dotted](7) -- (8);
\draw (8) -- (9);
\end{scope}
\end{tikzpicture}
$$ 
The only affine diagrams of this shape are $F_4^{(1)}$ and $E_6^{(2)}$. If $\Pia$ is of type $F_4^{(1)}$ then $\a_p$ has label 3, while if $\Pia$ is of type $E_6^{(2)}$, then $\a_p$ has label 2. Since in both cases the corresponding automorphism $\s$ is not an involution, it follows that $\Pi_0$ must contain a unique component.

Suppose that $t=2$ and that $\Pi_0$ is connected. By Lemma \ref{N(w*p)} together with Remark \ref{oss:max-cardinality}, it follows that $\{\Upsilon(\eta_1),\Upsilon(\eta_2)\}$ is an orthogonal subset of maximal cardinality in $\Phi(\Pi_0)^+_1$, and that $||\eta_1|| = ||\eta_2||$. Therefore, if $\Pi_\calS$ is not affine, it must be of type $A_3$. Suppose that this is the case; then by (\ref{sameedges}) $\a_p$ is connected to $\a_{\Pi_0}$ by a single edge. On the other hand, by Remark \ref{oss:max-cardinality} again, the symmetric space corresponding to the Hermitian pair $(\Pi_0, \gra_{\Pi_0})$ has rank 2. Up to an automorphism of $\Pi_0$, by Table \ref{table:hermitian-pairs} the possibilities for the pair $(\Pi_0, \gra_{\Pi_0})$ are the following: $(A_n,\a_2)$, $(B_n,\a_1)$, $(D_n,\a_1)$, $(D_5,\a_5)$, $(E_6,\a_1)$. But then by obvious considerations it follows that $\Pia$ also of finite type, a contradiction.

It remains to check the case when $t=1$. If $\Pi_\calS$ is of finite type, then it is of type $A_2$, $C_2$, or $G_2$. Since $\{\Upsilon(\eta_1)\}$ is a set of orthogonal roots of maximal cardinality in $\bigcup_{\Sigma}\Phi(\Sigma)^+_1$, it follows that $\Pi_0$ is connected, and the Hermitian symmetric variety corresponding to the pair $(\Pi_0, \a_{\Pi_0})$ has rank one, therefore $\Pi_0$ is of type $A$ and $\a_{\Pi_0}$ is an extremal root in $\Pi_0$.
%$$
%\begin{tikzpicture}[double distance=2pt,>=stealth] 
%\matrix [matrix of math nodes,row sep=1cm,column sep=.7cm]
%{%|(11)|\circ&&&[-1cm]&[-0.5cm]&\\
%|(22)|\circ&|(23)|\circ&|(24)|\circ& |(25)|\bullet.\\
%};
%\begin{scope}[every node/.style={midway,auto},thick]
%\draw (22) -- (23);
%\draw[dotted] (23) -- (24);
%\draw (24)--(25);
%\end{scope}
%\end{tikzpicture}
%$$
By (\ref{sameedges}), $\a_p$ is connected to $\a_{\Pi_0}$ by  the same number of edges that connect $-\a_p$ to $\eta_1$. Therefore, if $\Pi_\calS$ is of finite type, then $\Pia$ would be of finite type as well, which is absurd.
\end{proof}

For the reader's convenience we list here all possible diagrams for the affine root system $\Pi_\calS$, with the corresponding labels.

\begin{table}[h]
\begin{center}
\addtolength{\tabcolsep}{-1pt}
\begin{tabular}{ccccc}
\begin{tikzpicture}[double distance=2pt,>=stealth] 
\matrix [matrix of math nodes,row sep=.5cm,column sep=.5cm]
{&|(11)|\circ&\\
|(21)|\circ&|(22)|\circ&|(23)|\circ\\
&|(31)|\circ&\\
};
\node [above] at (11) {\tiny $1$};
\node [above] at (11) {\tiny $1$};
\node [below] at (21) {\tiny $1$};
\node(x) [below] at (22) {\tiny $2$};
\node [below] at (23) {\tiny $1$};
\node [below] at (31) {\tiny $1$};
\begin{scope}[every node/.style={midway,auto},thick]
\draw (11) -- (22);
\draw (21) -- (22);
\draw (31)--(x);
\draw (23)--(22);
\end{scope}
\end{tikzpicture}
&
\begin{tikzpicture}[double distance=2pt,>=stealth] 
\matrix [matrix of math nodes,row sep=.5cm,column sep=.5cm]
{&|(11)|\circ&\\
|(21)|\circ&|(22)|\circ&|(23)|\circ\\
&|(31)|&\\
};
\node [above] at (11) {\tiny $1$};
\node [below] at (21) {\tiny $1$};
\node [below] at (22) {\tiny $2$};
\node [below] at (23) {\tiny $2$};
\node [below] at (31) {\phantom{\tiny $1$}};
\begin{scope}[every node/.style={midway,auto},thick]
\draw (11) -- (22);
\draw (21) -- (22);
\draw[double,->] (22)--(23);
\end{scope}
\end{tikzpicture}
&
\begin{tikzpicture}[double distance=2pt,>=stealth] 
\matrix [matrix of math nodes,row sep=.5cm,column sep=.5cm]
{&|(11)|&\\
|(21)|\circ&|(22)|\circ&|(23)|\circ\\
&|(31)|&\\
};
\node [below] at (21) {\tiny $1$};
\node [below] at (22) {\tiny $1$};
\node [below] at (23) {\tiny $1$};
\node [below] at (31) {\phantom{\tiny $1$}};
\begin{scope}[every node/.style={midway,auto},thick]
\draw [double,<-] (21) -- (22);
\draw[double,->] (22)--(23);
\end{scope}
\end{tikzpicture}
& 
\begin{tikzpicture}[double distance=2pt,>=stealth] 
\matrix [matrix of math nodes,row sep=.5cm,column sep=.5cm]
{&|(11)|&\\
|(21)|\circ&|(22)|\circ&|(23)|\circ\\
&|(31)|&\\
};
\node [below] at (21) {\tiny $1$};
\node[below] at (23) {\tiny $3$};
\node [below] at (22) {\tiny $2$};
\node [below] at (31) {\phantom{\tiny $1$}};
\begin{scope}[every node/.style={midway,auto},thick]
\draw(21) -- (22);
\draw(6pt,3pt)-- +(11pt,0);
\draw(22) -- (23);
\draw(6pt,-3pt)-- +(11pt,0);
\draw(21pt,0)--(17pt,-3pt);
\draw(21pt,0)--(17pt,3pt);
\end{scope}
\end{tikzpicture}
& 
\begin{tikzpicture}[double distance=2pt,>=stealth] 
%\matrix [matrix of math nodes,row sep=.5cm,column sep=.5cm]
%{&|(11)|&\\
%|(21)|\circ\,&|(22)|\circ\\
%&|(31)|&\\
%};
\matrix [matrix of math nodes,row sep=.5cm,column sep=.5cm]
{
|(21)|\circ\,&|(22)|\circ\\
&|(31)|&\\
};
\node[below] at (21) {\tiny $2$};
\node [below] at (22) {\tiny $1$};
\node [below] at (31) {\phantom{\tiny $1$}};
\begin{scope}[every node/.style={midway,auto},thick]
\draw[double](21) -- (22);
\draw(-9pt,7pt) -- +(10pt,0);
\draw(-9pt,14pt) -- +(10pt,0);
\draw(-15.5pt,10.5pt) -- (-9pt,7pt);
\draw(-15.5pt,10.5pt) -- (-9pt,14pt);
\end{scope}
\end{tikzpicture} \\
$D_4^{(1)} \hspace{35pt}$ & $B_3^{(1)} \hspace{35pt}$ & $D_3^{(2)} \hspace{35pt}$ & $G_2^{(1)} \hspace{35pt}$ & $A_2^{(2)} \hspace{35pt}$ \\
& & & & \\
\end{tabular} 
\addtolength{\tabcolsep}{1pt}
\caption[font={scriptsize,italic}]{Affine root systems corresponding to non-spherical orbits in $G_0 \aa_p$} \label{diagramsns}
\end{center}
\end{table}

\begin{remark}	\label{oss:somma4}
%\footnote{Ho isolato questa osservazione che viene usata molto spesso.}
Let $\Xi$ be one of the affine diagrams of Table \ref{diagramsns}, let $k_\Xi$ be the integer such that $\Xi$ is of type $X^{(k_\Xi)}_{r}$ ($k_\Xi \in \{1,2\}$) and let $\a \in \Xi$ be the unique long simple root which is connected to all other simple roots. Let $\xi\in\Xi$ and let $a_{\Xi,\xi}$ be the corresponding label in $\Xi$, then for all $\xi \in \Xi$, we have $k_\Xi a_{\Xi,\xi} = |\langle \a, \xi^\vee \rangle|$. In particular,  the equality
$$k_\Xi(\sum_{\xi \neq \a} a_{\Xi,\xi}) = 4$$ 
holds. If $\calS$ is an orthogonal subset of $\calC^1_\s \setminus \{\gra_p\}$ of maximal cardinality and $\xi \in \Pi_\calS$, then we will denote $k_{\Pi_\calS}$ and $a_{\Pi_\calS,\xi}$ simply by $k_\calS$ and $a_{\calS,\xi}$.
\end{remark}

%Let $\calS = \{\eta_1, \ldots, \eta_t\}$ be an orthogonal subset of $\calC^1_\s \setminus \{\gra_p\}$ of maximal cardinality. If $\a\in\Pi_\calS$, let $a_{\calS,\a}$ be the corresponding label in $\Pi_\calS$, let $k_\calS$ be the integer such that $\Pi_\calS$ is of type $X^{(k_\calS)}_{r}$ ($k_\calS \in \{1,2\}$). Notice that in all possible cases we have $k_\calS(\sum_{i=1}^t a_{\calS,\eta_i}) = 4$ and $k_\calS a_{-\a_p,\calS} = 2$. 

%Recall that a $\mathfrak{sl}(2)$-triple $\{x,h,y\}$ in $\gog = \gog_0 \oplus \gog_1$ is called normal if $x,y \in \gog_1$ and $h \in \gog_0$.
%\footnote{Ho aggiunto la definizione di sl(2)-tripla normale}
As in  Section \ref{s4}, we fix a weight vector $x_1^\mu \in \gog_1^\mu$ for all $\mu\in\D_1^+$. Recall that $\ov\a_p$ is the lowest weight in $\D_1$ and that $\Psi(\aa_p) = \{-\ov\eta \st \eta \in \calC^1_\s\}$. If $\calS \subset \calC^1_\s \setminus \{\gra_p\}$ is an orthogonal subset, we set
$$x_\calS = \sum_{\eta\in\calS} x_1^{-\ov \eta} \in \aa_p.$$
For all $\eta \in \calS$, we choose $y_1^{\ov\eta} \in \gog_1^{\ov\eta}$  so that $[x_1^{-\ov\eta},y_1^{\ov\eta}]= -\ov\eta^\vee$. Since the weights $\ov\eta$ are strongly orthogonal, setting $y_\calS = \sum_{\eta \in \calS} y_1^{\ov\eta}$ and $h_\calS = -\sum_{\eta \in \calS} \ov\eta^\vee$, we have that $\{x_\calS,h_\calS,y_\calS\}$ is a normal $\mathfrak{sl}(2)$-triple which contains $x_\calS$ as a nilpositive element. 
%\footnote{Ho tirato fuori dalla dimostrazione la definizione della $\mathfrak{sl}(2)$-tripla, e dato in generale la
%definizione di $x_\calS$ che veniva fatto sepratamente in un paio di dimostrazioni.}

\begin{theorem}\label{specialisnotspherical}
Suppose that $\g_0$ is semisimple and $\a_p$ is long and non-complex, and let $\calS$ be an orthogonal subset of $\calC^1_\s \setminus \{\gra_p\}$ of maximal cardinality. Then the orbit $G_0 x_\calS \subset \gog_1$ is not spherical. In particular, $G_0\aa_p$ is not spherical.
%, and $\s$ is an involution of type II.
\end{theorem}

\begin{proof}
Let $\calS = \{\eta_1, \ldots, \eta_t\}$ be an orthogonal subset of $\calC^1_\s \setminus \{\gra_p\}$ of maximal cardinality. Since $\sum_{\a\in\Pi_\calS }a_{\calS,\a} \a$ is an isotropic vector, it follows that $k_\calS (\sum_{\a\in\Pi_\calS}a_{\calS,\a} \a)$ is a multiple of $\d$. Since 
$$
	[k_\calS (\sum_{\a\in\Pi_\calS }a_{\calS,\a} \a) : \a_p] = [k_\calS (\sum_{i=1}^ta_{\calS,\eta_i} \eta_i):\a_p] - k_\calS a_{-\a_p,\calS} = 4-2 = 2,
$$
it follows that $k_\calS (\sum_{\a\in\Pi_\calS }a_{\calS,\a} \a) = k\d$, hence
\begin{equation}\label{etasumalfa}
\sum_{i=1}^t k_{\calS}a_{\calS,\eta_i} \eta_i-\a_p=k\d+\a_p.
\end{equation}
It follows moreover that $k_{\calS} a_{\calS,\eta_i} (\eta_i,\eta_i) = 2(\a_p,\eta_i)$, namely
\begin{equation}\label{aetavee}
k_{\calS}a_{\calS,\eta_i} = \langle \a_p,\eta_i^\vee\rangle.
\end{equation}

The equalities (\ref{etasumalfa}) and (\ref{aetavee}) show that 
$$
\sum_{i=1}^t \langle \a_p,\eta_i^\vee \rangle \ov\eta_i = 2\ov\a_p,
$$
hence $\sum_{i=1}^t \langle \a_p,\eta_i^\vee \rangle \langle \ov\eta_i,\a_p^\vee \rangle =4$. Since $\gra_p$ is long, we have that $1=\langle \eta_i,\a_p^\vee \rangle = \langle \ov\eta_i,\a_p^\vee\rangle$, therefore considering the normal $\mathfrak{sl}(2)$-triple $\{x_\calS, h_\calS, y_\calS\}$ we get
$$-\ov\a_p(h_\calS) = \sum_{i=1}^t \langle \a_p,\ov\eta_i^\vee \rangle = \sum_{i=1}^t\langle \a_p,\eta_i^\vee \rangle = 4.
$$
Considering the grading $\bigoplus_{i \in \mathbb Z} \gog_1(i)$ induced by $h_\calS$, we see that $\gog_1^{-\ov \a_p} \subset \g_1(4)$, therefore $G_0x_\calS$ is not spherical by \cite[Theorem 5.6]{Pa2}.
%%%%%%%%%
% Set $e=\sum_{i=1}^tx_{-\ov\eta_i}$. Since both $\eta_i$ and $\a_p$ are in $N(w_p)$ and $w_p\in\Wab$, we see that $\eta_i+\a_p$ is not a root. It follows from  \eqref{aetavee} that $\a_p+z\eta_i$ is a root if and only if $-k_{\calS}a_{\eta_i}^{ns}\leq z\le 0$, thus we see that $ad(x_{-\ov\eta_i})^z(x_{\ov\a_p})\ne 0$ if and only if $0\le z\le k_{\calS}a_{\eta_i}^{ns}$. It follows that
%$$ad(e)^4(x_{\ov\a_p})=C\prod_{i=1}^tad(x_{-\ov\eta_i})^{k_{\calS}a^{ns}_{\ov\eta_i}}(x_{\ov\a_p})=C'x_{-\ov\a_p}\ne0,
%$$
%for some nonzero constants $C$, $C'$.
%Thus $\g_1(4)$ is nonzero, hence $G_0e$ is not spherical.
\end{proof}

Let $\grS \subset \Pi_0$ be a component, and let $e_\grS = - \langle \a_p, \a_\grS^\vee \rangle$ be the number of edges connecting $\a_\grS$ with $\a_p$. As a corollary of the previous proof we get the following.

\begin{corollary} 	\label{cor:somma-Y(eta)}
%\footnote{Ho tirato fuori questo corollario dalla dimostrazione del Lemma \ref{tubet}, visto che poi questi fatti venivano richiamati spesso.}
Let $\calS = \{\eta_1, \ldots, \eta_t\}$ be an orthogonal subset of maximal cardinality of $\calC^1_\s \setminus \{\gra_p\}$ and let $\grS \subset \Pi_0$ be a component.
\begin{itemize}
 \item[i)] If $\Upsilon(\eta_i) \in \D(\grS)$, then $k_{\calS}a_{\calS,\eta_i} = \langle \a_p, \eta_i^\vee \rangle = e_\grS$.
 \item[ii)] Let $(k\d-2\a_p)_\Sigma$ be the orthogonal projection of $k\d-2\a_p$ onto $\h_\Sigma = \Span(h_\a \st \a\in\Sigma)$ and set $I_\Sigma=\{i \st \Upsilon(\eta_i)\in\D(\Sigma)\}$.  Then
$$
\sum_{i\in I_\Sigma}\Upsilon(\eta_i)=\frac{1}{e_\Sigma}(k\d-2\a_p)_\Sigma.
$$
\end{itemize}
\end{corollary}

\begin{proof}
i) This follows immediately by formulas (\ref{sameedges}) and (\ref{aetavee}).

ii) Since $k_\calS(\sum_{i=1}^t a_{\calS,\eta_i}) = 4$, formula (\ref{etasumalfa}) implies that
$$
\sum_{i=1}^tk_{\calS}a_{\calS,\eta_i}\Upsilon(\eta_i)=k\d-2\a_p.
$$
On the other hand $(k\d-2\a_p)_\Sigma = \sum_{i\in I_\Sigma}k_{\calS}a_{\calS,\eta_i} \Upsilon(\eta_i)$, therefore the claim follows by i).
\end{proof}

We already noticed at the beginning of the section that $(\Sigma,\a_\Sigma)$ is a Hermitian pair. More precisely, we have the following.

\begin{proposition} \label{tubet}
Let $\grS \subset \Pi_0$ be a component, then $(\Sigma,\a_\Sigma)$ is a Hermitian pair of tube type.
\end{proposition}

\begin{proof}
For all $\a\in\Sigma\setminus \{\a_\Sigma\}$, it clearly holds $(\a,k\d-2\a_p)=(\a,(k\d-2\a_p)_\Sigma)=0$. By Corollary \ref{cor:somma-Y(eta)} we get then the equality
$$
(\a,\sum_{i\in I_\Sigma}\Upsilon(\eta_i))=i(\a_\S,\a_\S)  \qquad \mbox{  for all } \a\in \Phi(\Sigma)^+_i,\quad i=0,1.
$$
By Lemma \ref{N(w*p)}, $\{\Upsilon(\eta_i) \st i\in I_\Sigma\}$ is an orthogonal set of maximal cardinality in $\Phi(\Sigma)^+_1$. By Remark \ref{oss:max-cardinality} we can then choose a set of positive roots for $\Phi(\Sigma)_0$ in such a way that $\{\Upsilon(\eta_i) \st i\in I_\Sigma\}$ is the corresponding set of strongly orthogonal roots in $\Phi(\Sigma)^+_1$. Therefore the claim follows immediately from (\ref{sumtube}). 
\end{proof}

In Theorem \ref{specialisnotspherical} we constructed a non-spherical orbit $G_0x_\calS \subset G_0 \aa_p$ starting from an orthogonal subset $\calS \subset \mathcal C^1_\s$ of maximal cardinality. We now show that this construction extends to any maximal orthogonal subset $\calS \subset \mathcal C^1_\s$, and it always gives
rise to the same $G_0$-orbit.

\begin{theorem} \label{minimalnonspherical}
Let $\calS \subset \mathcal C^1_\s \setminus \{\a_p\}$ be a maximal orthogonal subset.
\begin{itemize}
	\item[i)] For all $\a \in \Phi$ it holds $\a(h_{\calS}) = -2 \langle \a,\a_p^\vee \rangle$.
	\item[ii)] The orbit $G_0 x_{\calS} \subset \gog_1$ is not spherical, and $G_0x_\calS=G_0x_{\mathcal T}$ for all maximal orthogonal subset $\mathcal T \subset \mathcal C^1_\s \setminus \{\a_p\}$.
\end{itemize}
\end{theorem}

\begin{proof}
If $\Sigma \subset \Pi_0$ is a component, set $\calS_\Sigma = \Upsilon(\calS)\cap\Phi(\Sigma)^+_1$, a maximal orthogonal subset in $\Phi(\Sigma)^+_1$. Write explicitly $\calS_\Sigma = \{\gamma_1,\ldots,\gamma_s,\be_1,\ldots,\be_t\}$ with $\gamma_1, \ldots, \grg_s$ long roots and $\be_1, \ldots, \be_t$ short roots. Let $r_\grS$ be the rank of the Hermitian symmetric space corresponding to the pair $(\grS,\a_\Sigma)$ and let $\calS'_\Sigma = \{\gamma_1,\ldots,\gamma_{r_\grS}\}$ be an orthogonal subset of maximal cardinality of $\Phi(\grS)^+_1$ containing $\{\gamma_1,\ldots,\gamma_s\}$. Then $\calS' = \bigcup_\Sigma \Upsilon^{-1}(\calS'_\Sigma)$ is an orthogonal subset of $\mathcal C^1_\s \setminus \{\a_p\}$ of maximal cardinality.

%Let $\grS \subset \Pi_0$ be a component. 
By Proposition \ref{tubet}, the Hermitian pair $(\grS,\a_\Sigma)$ is of tube type, therefore by Lemma \ref{short:tube}, upon relabelling $\grg_{s+1}, \ldots, \grg_{r_\grS}$, we may assume that $\be_i=\half(\gamma_{s+2i-1}+\gamma_{s+2i})$ for all $i=1, \ldots, t$. By Corollary \ref{cor:somma-Y(eta)} it follows that
$$
\sum_{i=1 }^s\gamma_i+\sum_{j=1}^t 2\be_j=\sum_{i=1}^{r_\grS} \gamma_i = \frac{1}{e_\Sigma}(k\d-2\a_p)_\Sigma.
$$
Thus
$$
e_\Sigma(\sum_{i=1 }^s\Upsilon^{-1}(\gamma_i)+\sum_{j=1}^t 2\Upsilon^{-1}(\be_j)) = (k\d-2\a_p)_\Sigma+e_\Sigma r_\grS \a_p.
$$
Summing over all components $\grS \subset \Pi_0$, we get 
$$
\sum_{\eta\in\calS}e_\eta \eta = k\d-2\a_p+\sum_\Sigma e_\Sigma r_\Sigma \a_p,
$$
where for $\eta \in \calC^1_\s \setminus \{\a_p\}$ we define
$$
e_\eta = -\langle \a_p,\Upsilon(\eta)^\vee \rangle = \left\{
\begin{array}{cl}
e_\Sigma & \mbox{ if } \Upsilon(\eta) \in \Phi(\Sigma)^+_1 \mbox{ is a long root,} \\
2 e_\Sigma & \mbox{ if } \Upsilon(\eta) \in \Phi(\Sigma)^+_1 \mbox{ is a short root.}
\end{array}\right.$$

Notice that by Lemma \ref{N(w*p)} and Corollary \ref{cor:somma-Y(eta)} i) we have $e_\eta = \langle \a_p,\eta^\vee \rangle = k_\calS a_{\eta,\calS}$ for all $\eta \in \calS'$. Therefore by Corollary \ref{cor:somma-Y(eta)} ii) and Remark \ref{oss:somma4} we obtain
$$
\sum_\Sigma e_\Sigma r_\Sigma = \sum_{\eta\in\calS'}e_\eta = \sum_{\eta\in\calS'} k_{\calS'} a_{\eta, \calS'} = 4.
$$

On the other hand, by Lemma \ref{N(w*p)}, we have $e_\eta = \langle \a_p,\eta^\vee \rangle$  for all $\eta \in \calS$, therefore we get
\begin{equation}\label{sumeta}
 \sum_{\eta\in\calS} \langle \a_p,\eta^\vee\rangle \eta =  \sum_{\eta\in\calS} e_\eta \eta = k\d+2\a_p.
\end{equation}

We now conclude the proof by showing that the orbit $G_0 x_\calS$ does not depend on the maximal orthogonal subset $\calS \subset \calC^1_\s$, hence it is non-spherical by Theorem \ref{specialisnotspherical}. We identify the orbit $G_0 x_{\calS} \subset \gog_1$ by computing its weighted Dynkin diagram (see \cite[Section 9.5]{Col}).
%\footnote{Aggiunto referenza per i diagrammi di Dynkin pesati.}

By Lemma \ref{N(w*p)} we have $\langle \eta,\a_p^\vee \rangle =1$ for all $\eta\in\mathcal C^1_\sigma \setminus \{\a_p\}$. If $\a\in\Phi$, considering the normal $\mathfrak{sl}(2)$-triple $\{x_\calS, h_\calS, y_\calS\}$, it follows that
$$
 \a(h_{\calS})=-\sum_{\eta\in\calS} \langle \a,\eta^\vee \rangle \langle \eta,\a_p^\vee \rangle = -\frac{2}{||\a_p||^2}(\a,\sum_{\eta\in\calS} \langle \a_p,\eta^\vee \rangle \eta). $$
By (\ref{sumeta}) we get then
$$
  \a(h_{\calS})=-\frac{2}{||\a_p||^2}(\a,k\d+2\a_p)=-2 \langle \a,\a_p^\vee \rangle .
$$
Since the right hand side of previous equality does not depend on $\calS$, it follows that $G_0 x_{\calS}$ does not depend on $\calS$ either.
\end{proof}

%\section{Some results on gradings}\label{6}

%Let $\aa \subset \gog_1$ be a $B_0$-stable abelian subalgebra of $\gog$ and let $x_0 \in \aa$ be such that $\aa = \ol{B_0 x_0}$. By the theorem of previous section $\aa$ is not spherical if and only $\height(x) = 4$, in which case $\gog_0(4) = 0$.

We conclude this Section by showing that $G_0 \aa_p$ has complexity one. Recall that the complexity of an irreducible $G_0$-variety is defined as
$$
c_{G_0}(X) = \min_{x\in X} \mbox{ codim} B_0 x.
$$
In particular, the spherical $G_0$-varieties coincide with the $G_0$-varieties of complexity zero, and the complexity of a $G_0$-variety can be regarded as a measure of its non-sphericity.
%\footnote{Ho spostato qui la definizione di complessit\`a, forse non ne parlerei di queste cose nell'introduzione.}
By Theorem \ref{teo:B0-orbite} together with Theorem \ref{minimalnonspherical}, if $\calS \subset \mathcal C^1_\s$ is a maximal orthogonal subset, then $x_\calS$ is in the open $G_0$-orbit of $G_0\aa_p$. Since the complexity of a $G_0$-variety coincides with that of its $G_0$-stable open subsets, it follows in particular that  $c_{G_0}(G_0 x_\calS) = 1$ for all maximal orthogonal subset $\calS \subset \mathcal C^1_\s \setminus \{\a_p\}$.

We start by recalling and proving some general results about the gradings associated to nilpotent elements of small height in $\mZ_2$-graded Lie algebras.
%\footnote{mancano referenze}

\subsection{Remarks on the gradings associated to nilpotent elements of small height in $\g_1$.}\label{51} Let $x_1 \in \gog_1$ be a nilpotent element and let $\{x_1, h_0, y_1\}$ be a normal $\mathfrak{sl}(2)$-triple containing $x_1$. Denote $\gop_0 = \gog_0(\geq\!\!0)$, $\gou_0 = \gog_0(\geq\!\!1)$ and $\gol_0 = \gog_0(0)$, then $\gop_0$ is a parabolic subalgebra of $\gog_0$ with Levi factor $\gol_0$ and with nilradical $\gou_0$. Fix a Borel subalgebra $\b_0 \subset \g_0$ contained in $\gop_0$ and set $\gob_{00} = \gob_0 \cap \gog_0(0)$, a Borel subalgebra of $\gol_0$, so that $\gob_0 = \gob_{00} \oplus \gou_0$. Notice that $\gog_1(i)$ is $\gol_0$-stable for all $i \in \mZ$. 

Let $P_0$ be the parabolic subgroup of $G_0$ corresponding to $\gop_0$, $L_0 \subset P_0$ the Levi factor corresponding to $\gol_0$ and $U_0$ the unipotent radical of $P_0$, $B_0 \subset G_0$ the Borel subgroup corresponding to $\b_0$, and $B_{00} \subset L_0$ be the Borel subgroup of $L_0$ corresponding to $\gob_{00}$.
Recall that in this section the fixed points set $\g_0$ of the involution is assumed to be semisimple and we denote by $\a_p$ the simple root in $\Pia$ corresponding to $\sigma$.
\begin{proposition}\label{apweight}
Suppose that $\gog_0$ is semisimple with corresponding simple root $\gra_p \in \widehat\Pi$, and let $x_1 \in \gog_1$ be a nilpotent element with $\height(x_1) = m$. Then $\gog_1(m)$ is an irreducible $L_0$-module. If moreover $\gra(h_0) \geq 0$ for all $\gra \in \Phi_0^+$, the highest weight vector of $\g_1(m)$ is $x_1^{-\ov\gra_p}$. %\footnote{Supponiamo $m=4$ e $\gog_0(4) = 0$, allora ha sempre dimensione 1? Equivalentemente, ogni radice di $\gol_0$ deve essere ortogonale a $\gra_p$.}
\end{proposition}

\begin{proof}
Up to conjugating $x_1$ we may assume that $\gra(h_0) \geq 0$ for all $\gra \in \Phi_0^+$. Notice that $\gou_0$ acts trivially on $\gog_1(m)$, therefore every highest weight of $\gog_1(m)$ as a $\gol_0$-module is actually a highest weight for $\gog_1$ as a $\gog_0$-module. On the other hand, since $\gog_0$ is semisimple, $\gog_1$ is an irreducible $\gog_0$-module with highest weight vector $x_1^{-\ov\gra_p}$, therefore $\gog_1(m)$ is an irreducible $\gol_0$-module as well, with highest weight vector $x_1^{-\ov\gra_p}$.
\end{proof}

Assume furthermore that $\height(x_1) \leq 4$ and $\gog_0(4) = 0$. %\footnote{Mi sembra che $\gog_0(4) = 0$ lo dimostriamo solo nella sezione successiva, o mi sbaglio? Si pu\'o vedere facilmente nel caso di $\aa_p$? Se no si pu\`o spostare tutto questo come ultima sezione, stando attenti alla Prop. \ref{apweight}, che in ogni caso adesso \`e indipendente da tutto il resto.} 
For $i \geq 2$, set $\aa_i = \gog_1(\geq\!\!i)$. Notice that $\aa_i$ is a $B_0$-stable abelian subalgebra of $\gog$: indeed $\aa_i$ is $\gop_0$-stable, hence $P_0$-stable, and being $\height(x_1) \leq 4$ and $\gog_0(4) = 0$, it follows that $[\aa_i, \aa_i] \subset \gog_0(4) = 0$. It follows then by Theorem~\ref{teo:B0-orbite} that $\aa_i$ possesses finitely many $B_0$-orbits, parametrized by the orthogonal subsets of $\Psi(\aa_i)$. 

\begin{proposition}	\label{prop:spherical-modules}
Suppose that $\height(x_1) \leq 4$ and that $\gog_0(4) = 0$. Then $\gog_1(i)$ is a spherical $L_0$-module for all $i\geq 2$.
\end{proposition}

\begin{proof}
Let $i \geq 2$ and consider the $B_0$-stable subalgebra $\aa_i$. By Theorem~\ref{teo:B0-orbite}, there is $v_i \in \aa_i$ such that $\ol{B_0 v_i} = \aa_i$. Since $\aa_i = \gog_1(i) \oplus \aa_{i+1}$ and since $\aa_{i+1}$ is also $B_0$-stable, we may write $v_i = u_i + u_i'$, for some $u_i \in \gog_1(2)$ and $u_i' \in \aa_{i+1}$ with $u_i \neq 0$. Therefore
$$
	\gog_1(i) \oplus \aa_{i+1} = \aa_i = [\gob_0,v_i] \subset [\gob_{00},u_i] \oplus ([\gou_0,u_i] + [\gob_0,u_i']).
$$
Since $[\gob_{00},u_i] \subset \gog_1(i)$ and $[\gou_0,u_i] + [\gob_0,u_i'] \subset \aa_{i+1}$,  the equality $[\gob_{00},u_i] = \gog_1(2)$ follows. Therefore $\gog_1(2) = \ol{B_{00}u_i}$ is a spherical $L_0$-module.
\end{proof}

\subsection{The complexity of $G_0 \aa_p$.}
We now apply the results of previous subsection to compute the complexity of $G_0 \aa_p$. As we already noticed, it is enough to show that $c_{G_0}(G_0 x_\calS) = 1$ when $\calS \subset \calC^1_\s \setminus \{\a_p\}$ is an orthogonal subset of maximal cardinality. In particular, under the bijection $\Upsilon : \calC^1_\s \setminus \{\a_p\} \rightarrow \bigcup_\grS \Phi(\Sigma)^+_1$ of Lemma \ref{N(w*p)}, we may assume that $\a_\grS \in \Upsilon(\calS)$ for all components $\Sigma \subset \Pi_0$.
\par\noindent
Let $\calS \subset \calC^1_\s \setminus \{\a_p\}$ be an orthogonal subset of maximal cardinality, let $\{x_\calS, h_\calS, y_\calS\}$ be the corresponding normal $\mathfrak{sl}(2)$-triple, let $\gog = \bigoplus \g(i)$, and for $j=0,1$ set $\gog_j(i) = \gog_j \cap \g(i)$. We keep the notation of previous subsection.

Let $K_0 \subset L_0$ be the identity component of the stabilizer of $x_\calS \in \gog_1(2)$. Then $K_0$ is reductive and $\gog_1(2)$ is a $K_0$-orthogonal module, therefore by a theorem of Luna there exists a reductive subgroup $M \subset K_0$ and a $K_0$-stable open subset $Z \subset \gog_1(2)$ such that every $K_0$-orbit in $Z$ is isomorphic to $K_0/M$ (see \cite[Section 5]{Pa2}). Then by \cite[Theorem 5.4]{Pa2} the following formula holds:
\begin{equation} \label{eq:complexity-panyushev}
	c_{G_0}(G_0 x_\calS) = c_{L_0}(\g_1(2)) + c_M(\g_1(\geq \!\!3)).
\end{equation}

\begin{proposition}	\label{prop:g1(3)andg1(4)}
%In the previous notation, 
We have $\g_1(3) =\g_0(4)=\{ 0\}$ and $\g_1(4) = \gog_1^{-\ov \a_p}$ is the trivial one-dimensional representation of $(L_0,L_0)$.
\end{proposition}

\begin{proof}
By Theorem \ref{specialisnotspherical},  we have $\a(h_\calS) = -2 \langle \a, \a_p^\vee \rangle$ for all $\a \in \Phi$. In particular $\g(3)=\{0\}$. If $\g(4)^\a\ne\{0\}$, then  $\langle\a,\alpha_p^\vee\rangle=-2$. Since $\a_p$ is long, it follows that $\a=-\ov\a_p$. Since $\a_p$ is non-complex,   $\g_0(4)=\{0\}$ and 
$  \g_1(4)=\gog_1^{-\ov \a_p}$.

By definition, $L_0$ is the Levi subgroup of $G_0$ whose set of simple roots is
$$
	\Pi_{00} = \bigcup_\grS \{ \a \in \grS \st \a(h_\calS) = 0\}.
$$
Let $\grS \subset \Pi_0$ be a component. Recall that $\a_\grS \in \grS$ is the unique simple root non-orthogonal to $\a_p$. On the other hand by Theorem \ref{minimalnonspherical} we have $\a(h_\calS) = -2 \langle \a, \a_p^\vee \rangle$, therefore $\Pi_{00} = \bigcup_{\grS} (\grS \setminus \{\a_\grS\})$ and it follows that every simple root of $L_0$ is orthogonal to $-\ol\a_p$. This show that $\g_1(4) = \g_1^{-\ol \a_p}$ is the trivial one-dimensional representation of $(L_0,L_0)$.
\end{proof}

\begin{corollary}\label{complexity}
Let $\calS \subset \calC^1_\s \setminus \{\a_p\}$ be an orthogonal subset of maximal cardinality, then $c_{G_0}(G_0 x_\calS) = 1$. In particular,  $c_{G_0}(G_0 \aa_p) = 1$. 
\end{corollary}

\begin{proof}
 Proposition \ref{prop:spherical-modules}  implies that $c_{L_0}(\g_1(2)) = 0$, whereas  Proposition \ref{prop:g1(3)andg1(4)} shows that $\g_1(\geq \!\!3) = \g_1(4)$ is one-dimensional, and by Theorem \ref{minimalnonspherical} together with (\ref{eq:complexity-panyushev}) we get $1 \leq c_{G_0}(G_0 x_\calS) = c_M(\g_1(4)) \leq 1$.
\end{proof}

%%%%%%%%%%%%%%%%%%%%%%%%%%%%%%%%
%%%%%%%%%%%%%%%%%%%%%%%%%%%%%%%%
\section{Classification of $B_0$-stable subalgebras of $\g_1$}\label{cla}
%%%%%%%%%%%%%%%%%%%%%%%%%%%%%%%%
%%%%%%%%%%%%%%%%%%%%%%%%%%%%%%%%

In this section $\g$ is a semisimple Lie algebra and $\s$ is an (indecomposable) involution of $\g$. Theorem \ref{teo:B0-orbite} prompts us to study the orbits $G_0 x \subset \gog_0 \oplus \gog_1$ with $x$ of the form
\begin{equation}\label{xs}x_\calS = \sum_{\gamma\in\calS_0} x_0^{\gamma} + \sum_{\gamma\in\calS_1} x_1^{\gamma},\end{equation}
where $\calS_0 \subset \Phi_0$, $\calS_1 \subset \Phi_1$, $\calS_0\cap\calS_1=\emptyset$,  and $\calS_0 \cup \calS_1$ is  a set of strongly orthogonal weights in $\D=\D_0 \cup \D_1$. We denote by $\calS$ the disjoint union of $\calS_0$ and $\calS_1$. %(notice that $\calS_0$ and $\calS_1$ might intersect).
 Notice that all $x_{\mathcal S}$ are nilpotent: indeed setting $y_{\mathcal S} = \sum_{\gamma\in\calS_0} y_0^{-\gamma} + \sum_{\gamma\in\calS_1} y_1^{-\gamma}$ and $h_{\mathcal S} = \sum_{\gamma\in\calS_1} \gamma^\vee + \sum_{\gamma\in\calS_2} \gamma^\vee$ we get a $\mathfrak{sl}(2)$-triple $\{x_\calS,h_\calS,y_\calS\}$. 
 %\footnote{Ho cambiato un po' questo periodo}.

Let $m$ be the height of $x_\calS$. Since $h_\calS \in\h_0$, we can choose a weight $\a\in\D$ such that $\g^\a \subset \g(m)$, namely such that
$$
\gra(h_\calS) = \sum_{\gamma\in\calS} \langle\a,\gamma^\vee\rangle = m.
$$
Set $\calS^+(\a)=\{\gamma\in\calS\st \langle\a,\gamma^\vee\rangle > 0\}$, so that 
\begin{equation}\label{Gammaaffine}
\sum_{\gamma\in\calS^+(\a)} \langle\a,\gamma^\vee\rangle \geq m.
\end{equation}

Define $\hat\a=\a$ if $\a \in \Phi_0$, and $\hat \a=\d'+\a$ if $\a\in\Phi_1\setminus\Phi_0$. Choose for each $\gamma\in \calS^+(\a)$ a root $\hat\gamma\in\Da$ such that $\overline{\hat\gamma}=\gamma$, and define $\widehat{\calS^+(\a)}=\{\hat\gamma\st \gamma\in\calS^+(\a)\}$ and $\Pi_{\calS,\a} = \widehat{\calS^+(\a)}\cup\{-\hat \a\}$. As $\calS^+(\a)\cup\{-\a\}\subset \D$, we have that $\Pi_{\calS,\a} \subset \Da^{re}$, so the matrix $A(\calS,\a) = (\langle\be,\xi^\vee\rangle)_{\be,\xi\in\Pi_{\calS,\a}}$ is a generalized Cartan matrix, which is of finite or affine type by Lemma \ref{affine type}.

\begin{lemma}	\label{lemma:height 4}
Let $\calS\subset \D$ be a strongly orthogonal subset, and let $\a\in \D$ be such that $\a(h_\calS) = \height(x_\calS)$. Then the following statements hold.
 \begin{itemize}
 \item[i)] \label{charsph1} $\height(x_\calS)$ is less than or equal to the degree of  $-\hat\a$ in $\Pi_{\calS,\a}$. In particular $\height(x_\calS)\leq 4$.
 \item[ii)] If $G x_\calS$ is not spherical, then $\height(x_\calS) = 4$ and $\Pi_{\calS,\a}$ is of affine type, in which case its diagram is one of those listed in Table \ref{diagramsns}.
 \end{itemize}
 \end{lemma}

\begin{proof}
i) If $\grg \in \calS^+(\a)$, notice that $\langle\a,\gamma^\vee\rangle = \langle \hat\a,\hat\gamma^\vee\rangle$, and that this number   is less  or equal to the number of edges connecting $-\hat \a$ with $\hat \gamma$. Therefore the claim follows by formula (\ref{Gammaaffine}), by observing  that  the degree of  any node in a finite or affine diagram is at most $4$.

ii) By \cite[Theorem 3.1]{Pa1}, if $G x_\calS$ is not spherical then $\height(x_\calS)\geq 4$, hence $\height(x_\calS)= 4$ by i). Since in a Dynkin diagram of finite type any node has degree at most $3$, it follows that $\Pi_{\calS,\a}$ is affine. Moreover, if $\langle\hat\a,\hat\gamma^\vee \rangle$ is less than the number of edges connecting $\hat\a$ and $\hat\gamma$ for some $\grg \in \calS^+(\a)$, then $\sum_{\gamma\in\calS^+(\a)}\langle \hat\a,\hat\gamma^\vee \rangle<4$. Thus, for all $\grg \in \calS^+(\a)$, $\langle\hat\a,\hat\gamma^\vee\rangle$ equals the number of edges connecting $\hat\a$ and $\hat\gamma$, and $\hat\a$ is long in $\Pi_{\calS,\a}$. It follows that the diagram $\Pi_{\calS,\a}$ is one of the affine diagrams listed in Table \ref{diagramsns}.
\end{proof}

In the next result we use the main idea of Proposition 2.2 of \cite{PR}.

\begin{lemma}\label{gzero4=0}
Let $\calS\subset \D$ be a strongly orthogonal subset and suppose that $\calS \subset \Psi(\aa)$ for some $\aa\in\mathcal I^\s_{ab}$. Let $\a\in\D$ be such that $\a(h_\calS) = \height(x_\calS)$, then $\a\in\Phi_1\setminus \D_0$ and
\[
\sum_{\gamma\in\calS^+(\a)}\langle\a,\gamma^\vee\rangle \hat\gamma=k\d+2\hat\a.
\]
In particular, $\height_0(x_\calS) \leq 3$ and $\sum_{\gamma\in\calS^+(\a)} \langle \be, \gamma^\vee\rangle =2 \langle \be,\a^\vee\rangle$ for all $\be\in\Phi \cup \{0\}$.
\end{lemma}

\begin{proof}
Since $\calS^+(\a) \subset \D_1$ we can assume that $\hat \gamma=\d'+\gamma$ for all $\gamma \in\calS^+$.
%Arguing with Remark \ref{oss:somma4} 
As in the proof of Theorem \ref{specialisnotspherical}, we find that, if $\Pi_{\calS,\a} = \Xi$, then $k_{\Xi}(\sum_{\xi\in\Xi} a_{\Xi, \xi}\xi)$ is an isotropic vector, hence it is a multiple of $\d$, say
\begin{equation}\label{occurrences}
k_{\Xi}(\sum_{\xi\in\Xi} a_{\Xi,\xi} \xi) = s\d.
\end{equation}
The coefficient $s$ can be computed by counting the occurrences of roots in $\D_1^+$  in the left hand side of (\ref{occurrences}). It follows   that $s=2k$ if $\a\in\D_0$, and $s=k$ if $\a \in \D_1 \setminus \D_0$.
% \footnote{Non capisco perch\'e.} We show that the first case never happens.

Define a multiset $\{\hat\gamma_i\st i=1,\cdots,4\}$ by listing every $\hat\gamma \in \widehat{\calS^+(\a)}$ with multiplicity $\langle\hat\a,\hat\gamma^\vee \rangle = k_\Xi a_{\Xi,\hat\gamma}$ (see Remark \ref{oss:somma4}). Set $\beta=k\d+\hat\a$. Since $\hat\a=\a\in\D_0$, it follows that $\grb$ is a root. Notice that $\beta-\hat\gamma_i-\hat\gamma_j$ is a root for all $i\neq j$. If indeed $\hat\gamma_i=\hat\gamma_j$ for some $i \neq j$, then $\langle\hat\a,\hat\gamma_i^\vee \rangle \ge 2$, hence $\langle \beta,\hat\gamma_i^\vee \rangle \geq 2$ and $\beta-2\hat\gamma_i$ is a root. If instead $\hat\gamma_i \ne \hat\gamma_j$, then $\langle \hat\a,\hat\gamma_i^\vee \rangle > 0$ and $\langle \hat\a-\hat\gamma_i,\hat\gamma_j^\vee \rangle > 0$, so $\beta-\hat\gamma_i-\hat\gamma_j$ is either a root or zero. On the other hand $\hat\gamma_i+\hat\gamma_j$ cannot be a root because $\aa$ is abelian, therefore $\beta-\hat\gamma_i-\hat\gamma_j$ is a root also in this case.

Since
  $
  \sum_{i=1}^4\hat\gamma_{i} = \sum_{\gamma\in\calS^+(\a)} k_\Xi a_{\Xi,\hat\gamma} \hat\gamma = 2k\d+2\hat\a = 2\beta
  $, we have
 $$
 \sum_{i<j}(\beta-\hat\gamma_{i}-\hat\gamma_{j}) = 6\beta - 3\sum_{i=1}^4\hat\gamma_{i}=0,
 $$
thus $\beta-\hat\gamma_{i}-\hat\gamma_{j}$ is a positive root for some $i <j$. Since $\hat\gamma_{i}\in\Dap_1$ for each $i$, $\beta-\hat\gamma_{i}-\hat\gamma_{j}$ is a root in $\Dap_0$. Since $\langle \beta,\hat\gamma_i^\vee \rangle>0$ and $\height_\s(\beta-\hat\gamma_i)=1$, $\beta-\hat\gamma_{i} \in \widehat \Phi_1$. Since $\beta-\hat\gamma_{i}=\hat\gamma_{j}+(\beta-\hat\gamma_{i}-\hat\gamma_{j})$, we see that  $\overline{\beta-\hat\gamma_{i}}\in \Psi(\aa)$. Since $\overline{\beta-\hat\gamma_{i}}$ and $\gamma_{i}$ are both in $\Psi(\aa)$ and $\beta-\hat\gamma_{i}+\hat\gamma_{i}=\beta$, we reach a contradiction since $\aa$ is abelian.

To prove the last claim, notice that $\a$ is long in $\Pi_{\calS,\a}$, hence $\langle\gamma,\a^\vee \rangle=1$. If $\be\in\Phi_0 \cup \Phi_1 \cup\{0\}$, we get then the equality
\[\sum_{\gamma\in\calS^+(\a)} \langle \be, \gamma^\vee\rangle = \frac{2}{||\a||^2}\sum_{\gamma\in\calS^+(\a)} \langle \a,\gamma^\vee \rangle (\be, \hat\gamma) = 2 \frac{(\be,k\d+2\hat\a)}{||\a||^2} = 2 \langle \be, \a^\vee \rangle. \]
\end{proof}

As a consequence of Lemma \ref{gzero4=0}, we get the following result.

%\begin{corollary}
%Let $\aa$ be a $\b_0$-stable abelian subalgebra of $\g_1$, let $x\in \aa$. In the grading defined by a normal triple  containing $x$ as a nilpositive element, we have $\g_0(4)=0$.
%\end{corollary}

%\begin{proof}
%Since the elements $x_\calS$ with $\calS \subset \Phi_1^+$ give a complete set of representatives for the $B_0$-orbits in $\aa$, the claim follows since the height of an element is invariant under the %adjoint action.
%\end{proof}

%\begin{corollary} 	\label{cor:altezza 4}
%Let $\aa$ be a $B_0$-stable abelian subalgebra of $\g_1$, and let $x \in \aa$, then $Gx$ is spherical if and only if $G_0 x$ is spherical if and only if $\mathrm{ad}(x)^4_{|\g_1}=0$. In particular, $G_0 \aa$ %is spherical if and only if $G\aa$ is spherical. 
%\end{corollary}

%\begin{corollary}\label{73}
%Let $x\in\g_1$ be a nilpotent element. There exists a Borel subgroup $B_0 \subset G_0$ and a $B_0$-stable subalgebra $\aa$ in $\g_1$ with $x \in \aa$ if and only if 
%$$
%\height(x)\le 4 \text{ and }\g(4)=\g_1(4)
%$$
%(in the grading defined by a normal triple  containing $x$ as a nilpositive element).
%\end{corollary}

\begin{corollary}\label{73}
Let $\aa \in \mathcal I^\s_{ab}$ and let $x \in \aa$, then $\height(x) \leq 4$ and $\height_0(x) \leq 3$. In particular, $Gx$ is spherical if and only if $\height_1(x)\leq 3$,  if and only if $G_0x$ is spherical.
\end{corollary}

\begin{proof}
By Theorem \ref{teo:B0-orbite}, acting with $B_0$ we may assume that $x = x_\calS$ for some orthogonal subset $\calS \subset \Psi(\aa)$. Then by Lemma \ref{lemma:height 4} we get $\height(x)\le 4$, and by Lemma \ref{gzero4=0} we get $\height_0(x) \leq 3$. The last claim follows by \cite[Theorem 5.6]{Pa2}.
\end{proof}

If in previous corollary we take $x$ in the open $B_0$-orbit of $\aa$, then we get the following.

\begin{corollary}\label{74}
Let $\aa \in \mathcal I^\s_{ab}$, then $\aa$ is $G$-spherical if and only if it is $G_0$-spherical.
\end{corollary}

%The previous corollary was already noticed in the introduction of \cite{Pa4}, but no proof was given, and no conceptual proof was known. \\

Recall that $\Pi_1$ contains at most two elements, and that if $\Pia$ is simply laced, then the real roots of $\Da$ are regarded as long. The next result has been proved in \cite{Pa4} as a consequence of a case-by-case inspection. We provide here a conceptual proof that follows from Lemma \ref{gzero4=0} and the results of Section \ref{51}. Note also that Theorem \ref{Panyushev} includes Theorem 2.3 of \cite{PR}.

\begin{theorem}\label{Panyushev}
There exists $\aa\in\mathcal I^\s_{ab}$ such that $G_0\aa$ is not spherical if and only if $\Pi_1=\{\a_p\}$ and $\a_p$ is long and non-complex.
\end{theorem}

\begin{proof}
If $\Pi_1=\{\a_p\}$ with $\a_p$ long and non-complex then by Theorem \ref{specialisnotspherical} the special $B_0$-stable subalgebra $\a_p$ gives rise to a non-spherical variety $G_0\aa_p$.

Let now $\aa\in\mathcal I^\s_{ab}$ and suppose that $G_0\aa$ is not spherical. By Theorem \ref{teo:B0-orbite} there is an orthogonal subset $\calS \subset \Psi(\aa)$ such that $G_0 x_\calS$ is not spherical, and by Lemma \ref{lemma:height 4} we get $\height(x_\calS) = 4$. Fix $\a\in\D$ such that $\a(h_\calS) = 4$, then $\a\in\D_1\setminus \D_0$ by Lemma \ref{gzero4=0}. Set $\hat\a=\d'+\a\in\Da_1$.

If $\Pi_1 = \{\a_i, \a_j\}$ consists of two distinct elements, then $k=1$. Moreover, since $\hat\a\in\Dap_1$, we can assume $[\hat\a:\a_i]=1$ and $[\hat\a:\a_j]=0$. By Lemma \ref{gzero4=0} we have
$$
\sum_{\gamma\in\calS^+(\a)}\langle \a,\gamma^\vee \rangle \hat\gamma-2\hat\a=\d.
$$
Being $(\hat\gamma,\hat\a)>0$, for all $\grg \in \calS^+(\a)$ relation  $\hat\gamma-\hat\a\in\Da\cup\{0\}$ holds. Therefore $[\hat\gamma:\a_i]=1$ for all $\gamma \in \calS^+(\a)$, and we get $[\d:\a_i]=2$ which is absurd.

Thus $\Pi_1=\{\a_p\}$ consists of a single element, and $\g_0$ is semisimple. By Proposition \ref{apweight}, we can choose $\a=w(-\ov\a_p)$ with $w\in W_0$. Since $\a\in\D_1\setminus \D_0$, we see that $\a_p$ cannot be complex. Suppose that $\a_p$ is short. Then also $\a$ is short, and there must be a component $\Sigma \subset \Pi_0$ such that $\theta_\Sigma$ is long. Since $(\theta_\Sigma,\a_p)=(\a_\Sigma, \a_p)<0$, it follows that $\langle \theta_\Sigma,\a_p^\vee \rangle \leq-2$, and setting $\be=w(\theta_\Sigma)$ we get $\langle \be,\a^\vee \rangle \ge2$.
Since $\a(\sum_{\gamma\in\calS^+(\a)}\gamma^\vee) = \height(x_\calS) = 4$, it follows that the element $x_{\calS^+(\a)} = \sum_{\gamma\in\calS^+(\a)}x_\gamma$ is still in $\aa$, and being $\height(x_{\calS^+(\a)}) \geq 4$ its $G_0$-orbit is still non-spherical by Corollary \ref{73}.
Therefore we can assume that $\calS = \calS^+(\a)$. By Lemma \ref{gzero4=0} we get then
$$
\be(\sum_{\gamma\in\calS}\gamma^\vee)=\be(\sum_{\gamma\in\calS^+(\a)}\gamma^\vee)=2 \langle \be,\a^\vee \rangle \ge 4.
$$
As $\grb \in \Phi_0$ and $\g_0(i)=0$ for $i>3$, this is absurd. Therefore $\a_p$ must be long.
\end{proof}

Assume now that $\Pi_1=\{\a_p\}$ with $\a_p$ long and non-complex. We  give a classification of the subalgebras $\aa\in\mathcal I^\s_{ab}$ such that $G_0\aa$ is non-spherical. By Proposition \ref{lemmahermitiano}, for any component $\grS \subset \Pi_0$ there is a unique maximally orthogonal antichain $\mathcal A_\Sigma$ in $\Phi(\Sigma)^+_1$. Regarding $\mathcal C^1_{\s}$ as a poset w.r.t. the dominance order, it is clear from Lemma \ref{N(w*p)}  that $\mathcal C^1_{\s}\setminus\{\a_p\}$ contains a unique maximally orthogonal antichain $\calA$, namely
$$
 \mathcal A=\bigcup_\Sigma\Upsilon^{-1}(\mathcal A_\Sigma).
$$

For $\grG \subset \mathcal C^1_{\s}$, we set $\grG^{\leq_0} = \{\xi\in \mathcal C^1_{\s} \st \xi \leq_0 \eta \  \mbox{ for some } \eta \in \grG\}$.

\begin{lemma}\label{overline w}
There is $\overline w\in\Wab$ such that $N(\overline w) = \calA^{\leq_0}$.
\end{lemma}

\begin{proof} 
We show that, if $\zeta,\xi\in\Dap$ are such that $\zeta+\xi\in\calA^{\leq_0}$, then exactly one among $\zeta$ and $\xi$ is in $\calA^{\leq_0}$.
Since $\calA^{\leq_0} \subset \mathcal C^1_\s = N(w_p)$,  then exactly one among $\zeta$ and $\xi$ (say $\zeta$) is in $\mathcal C^1_\s$. Since $\zeta+\xi \in \calA^{\leq_0}$,  then   $\zeta \in \calA^{\leq_0}$. This implies that both $\calA^{\leq_0}$ and its complement are closed under root addition. It follows that there is $\overline w\in\Wa$ such that $N(\overline w) = \calA^{\leq_0}$. Since $N(\overline w)\subset N(w_p)$, it is clear that $\overline w\in\Wab$.
\end{proof}

Let $\ol \aa=\Theta(\ol w)$ (see Proposition \ref{imrn}). 

\begin{theorem}\label{MT}
Suppose that $\Pi_1=\{\a_p\}$ with $\a_p$ long and non-complex, and let $\aa\in\mathcal I^\s_{ab}$. Then $G_0\aa$ is not spherical if and only if $\{-\ov\eta\st \eta\in\mathcal A\}\subset \Psi(\aa)$, if and only if $\ol \aa\subset \aa$.
\end{theorem}

\begin{proof}
Suppose that $G_0\aa$ is non-spherical, and let $\calS \subset \Psi(\aa)$ be an orthogonal subset such that $G_0 x_\calS$ is not spherical. Then by Lemma \ref{lemma:height 4} we have $\height(x_\calS) = 4$, and by Corollary \ref{73} there is $\a\in\D_1$ such that $\g_1(4)^\a \neq 0$, namely $\a(h_\calS) = \sum_{\grg \in \calS} \langle \a, \grg^\vee \rangle = 4$. We may assume that $\a$ is maximal w.r.t. $\leq_0$ among the weights of $\g_1(4)$. As in the proof of Theorem \ref{Panyushev}, we can assume that $\calS = \calS^+(\a)$. 

Suppose that $\a = -\ov\a_p$. 
% Recall that $-\ov\a_p$ is the highest weight of $\g_1$
We have that $\hat\a=\d'-\ov\a_p=-\a_p+2\d'=k\d-\a_p$. If $\gamma\in\calS$, since $\a_p$ is long we get
$$
\langle \d'-\gamma,(k\d+\a_p)^\vee \rangle = \langle \gamma,\a^\vee \rangle = 1,
$$
hence $\d'-\gamma\in\mathcal C^1_\s$.
Thus $\mathcal O=\{\d'-\gamma\st \gamma\in\calS\}$ is an orthogonal subset of $\mathcal C^1_\s$. Notice moreover that $\mathcal O$ is a maximal orthogonal subset in $\mathcal C^1_\s$: if indeed $\eta\in\mathcal O^\perp\cap\mathcal C^1_\s$, then $\Pi_{\calS,\a} \cup\{\eta\}$ gives rise to a generalized Cartan matrix which is neither finite nor affine, contradicting Lemma \ref{affine type}. Therefore $\Upsilon(\mathcal O)$ is a maximal orthogonal set in $\bigcup_\Sigma\Phi(\Sigma)^+_1$, and by Corollary \ref{cor:anticatena-sottosopra} it follows that $\bigcup_\Sigma\mathcal A_\Sigma \subset \Upsilon(\mathcal O)^{\leq_0}$. By Lemma \ref{N(w*p)}, it follows that $\mathcal A\subset \mathcal O^{\leq_0}\subset N(\Theta^{-1}(\aa))$ which in turns means that $N(\overline w)\subset N(\Theta^{-1}(\aa))$, or, equivalently, that $\overline\aa\subset\aa$.

Suppose that $\a \neq -\ol\a_p$. Then there exists $\be\in\Dp_0$ such that $\a+\be\in\D_0 \cup \Phi_1 \cup\{0\}$, and the maximality of $\a$ among the weights of $\gog_1(4)$ implies that $\sum_{\gamma\in\calS} \langle \a+\be,\gamma^\vee \rangle <4$. Therefore by Lemma \ref{gzero4=0} we get
\begin{equation}\label{sumlesszero}
2\langle \be,\a^\vee \rangle = \sum_{\gamma\in\calS} \langle \be,\gamma^\vee \rangle <0.
\end{equation}
  In particular there is $\gamma$ such that $(\beta, \gamma) < 0$ and, by Lemma~\ref{lem:adding-roots} it follows $(\beta, \gamma') \ge 0$ for all $\gamma' \in \calS \setminus \{\gamma\}$. 

Suppose that $(\be, \gamma')=0$ for all $\gamma' \in \calS$. Then $s_\be(\calS)$ is an orthogonal subset of $\Psi(\aa)$ and $G_0 x_{s_\beta(\calS)} = G_0 x_\calS$ is still not spherical, and $s_\be(\a)$ is a maximal weight in $\Phi_1$ w.r.t. $\leq_0$ such that $s_\be(\a)(h_{s_\be(\calS)}) = \height(x_{s_\be(\calS)}) = 4$. On the other hand by (\ref{sumlesszero}) we have $\a\le_0 s_\be(\a)$, therefore we may proceed inductively by replacing $x_\calS$ with $x_{ s_\be(\calS)}$ until either $\a = -\ol\a_p$ or we find a root $\grg'\in\calS$ such that $\langle \gamma',\be^\vee \rangle >0$. 

If $\a = -\ol\a_p$ then we are done, therefore we may assume that there are $\grg,\grg' \in \calS$ such that $(\grg,\be) < 0$ and $(\grg',\be) > 0$. Consider the set $\Pi_\be=\{\gamma,\be,-\gamma'\}$, then $A_\be=(\langle \nu,\xi^\vee\rangle )_{\nu,\xi\in\Pi_\be}$ is a generalized Cartan matrix, and by Lemma \ref{affine type} it is either of affine or of finite type. Identify $\Pi_\be$ with the corresponding Dynkin diagram; since $(\beta, \gamma') \ge 0$ for all $\gamma' \in \calS \setminus \{\gamma\}$, by (\ref{sumlesszero}) we have that $\langle \be,\gamma^\vee \rangle \leq -2$, so $\Pi_\be$ is not simply laced and $\gamma$ is a short node. If moreover $\langle \be,\gamma^\vee\rangle = -2$, then again by (\ref{sumlesszero}) it follows that $\langle \be,\gamma'^\vee\rangle =1$. With these conditions at hand, we see that the only possibilities for the diagram of $\Pi_\be$ are the following:
$$
\begin{tikzpicture}[double distance=2pt,>=stealth] 
\matrix [matrix of math nodes,row sep=.5cm,column sep=.5cm]
{
|(21)|\circ&|(22)|\circ&|(23)|\circ\\
};
\node [below] at (21) {$-\gamma'$};
\node [below] at (22) {$\be$};
\node [below] at (23) {$\gamma$};
\begin{scope}[every node/.style={midway,auto},thick]
\draw [double,->] (21) -- (22);
\draw[double,->] (22)--(23);
\end{scope}
\end{tikzpicture}\quad\begin{tikzpicture}[double distance=2pt,>=stealth] 
\matrix [matrix of math nodes,row sep=.5cm,column sep=.5cm]
{
|(21)|\circ&|(22)|\circ&|(23)|\circ\\
};
\node [below] at (21) {$-\gamma'$};
\node [below] at (22) {$\be$};
\node [below] at (23) {$\gamma$};
\begin{scope}[every node/.style={midway,auto},thick]
\draw (21) -- (22);
\draw[double,->] (22)--(23);
\end{scope}
\end{tikzpicture}
\quad
\begin{tikzpicture}[double distance=2pt,>=stealth] 
\matrix [matrix of math nodes,row sep=.5cm,column sep=.5cm]
{
|(21)|\circ\,&|(22)|\circ&|(23)|\circ\\
};
\node [below] at (21) {$-\gamma'$};
\node[below] at (23) {$\gamma$};
\node [below] at (22) {$\be$};
\begin{scope}[every node/.style={midway,auto},thick]
\draw(21) -- (22);
\draw(7pt,3pt)-- +(11pt,0);
\draw(22) -- (23);
\draw(7pt,-3pt)-- +(11pt,0);
\draw(22pt,0)--(18pt,-3pt);
\draw(22pt,0)--(18pt,3pt);
\end{scope}
\end{tikzpicture}
$$

The first case has to be discarded because, otherwise, the diagram of $\Pi_{\calS,\a}$ would have rank three with two nodes $\gamma, \gamma'$ satisfying $\frac{\Vert \gamma'\Vert^2}{\Vert \gamma\Vert^2}=4$, and this never occurs for the diagrams of Table \ref{diagramsns}. The second case also has to be discarded, otherwise by Lemma \ref{gzero4=0} it would follow
$$
2 \langle \be,\a^\vee \rangle = \sum_{\gamma\in\calS} \langle \be,\gamma^\vee \rangle =-1.
$$

Therefore the diagram of $\Pi_\be$ is of type $G_2^{(1)}$. This is possible only if $\widehat L(\g,\s)$ is of type $G_2^{(1)}$. Let $\Pia=\{\a_0,\a_1,\a_2\}$ be as in \cite[Table Aff 1]{Kac} (in particular,  $\a_2$ is short). Since $\be$ is long and belongs to $\Dp_0$, we have that $\be=\a_0$.
Moreover, $-\gamma'+2\be+3\gamma$ is isotropic so $-\hat\gamma'+2\be+3\hat\gamma=\d$, hence $-\hat\gamma'+3\hat\gamma=3\a_0+2\a_1+3\a_2$. Since $\hat\gamma',\hat\gamma\in\Dap_1$, we have that $[\hat\gamma:\a_0]\leq 1$ and $[\gamma':\a_0]\leq 1$. It follows that $\hat\gamma'=\a_1+x\a_2$ and $\hat\gamma=\a_0+\a_1+y\a_2$. From $[-\hat\gamma'+3\hat\gamma:\a_2]=3$ and $(\hat\gamma',\hat\gamma)=0$ we obtain that either $\hat\gamma'=\a_1,\hat\gamma=\a_0+\a_1+\a_2$ or $\hat\gamma'=\a_1+3\a_2,\hat\gamma=\a_0+\a_1+2\a_2$. In both cases one easily verifies that $\overline{\hat\gamma}$ cannot belong to $\Psi(\aa)$ with $\aa\in\mathcal I^\s_{ab}$. Hence we have obtained the desired contradiction. We conclude that $\a = -\ov\a_p$, and the proof is complete.
\end{proof}

\vskip10pt
\footnotesize{

\noindent{\bf J.G.}: Scuola Normale Superiore, Piazza dei Cavalieri 7,
                56126 Pisa,
                Italy;
{\tt jacopo.gandini@sns.it}

\noindent{\bf P.M.F.}: Politecnico di Milano, Polo regionale di Como, 
Via Valleggio 11, 22100 Como,
Italy; {\tt pierluigi.moseneder@polimi.it}

\noindent{\bf P.P.}: Dipartimento di Matematica, Sapienza Universit\`a di Roma, P.le A. Moro 2, 00185 Roma, Italy; {\tt papi@mat.uniroma1.it}
\end{document}